\documentclass[11pt, oneside]{amsart}

\usepackage[parfill]{parskip}
\usepackage{amssymb}
\usepackage{amsmath}
\usepackage{latexsym}
\usepackage{mathrsfs}
\usepackage{enumerate}
\usepackage{amsthm}
\usepackage{cite}
\usepackage{color}
\usepackage[foot]{amsaddr}

\newcommand{\E}{\mathbb{E}}
\newcommand{\p}{\mathbb{P}}
\newcommand{\opL}{\mathcal{L}}
\newcommand{\T}{\mathsf{T}}

\newtheorem{lemma}{Lemma}
\newtheorem{theorem}{Theorem}
\newtheorem{condition}{Condition}
\newtheorem{definition}{Definition}
\newtheorem{remark}{Remark}
\newtheorem{proposition}{Proposition}

\numberwithin{lemma}{section}
\numberwithin{theorem}{section}
\numberwithin{condition}{section}
\numberwithin{definition}{section}
\numberwithin{remark}{section}
\numberwithin{proposition}{section}

\title[Moderate Deviations for Systems of Slow-Fast Diffusions]{Moderate Deviations Principle for Systems of Slow-Fast Diffusions}
\author{Matthew R. Morse}
\author{Konstantinos Spiliopoulos}
\address{Department of Mathematics and Statistics \\
Boston University \\
Boston, MA 02215}
\email[Matthew R. Morse]{mrmorse@bu.edu}
\email[Konstantinos Spiliopoulos]{kspiliop@math.bu.edu}
\thanks{The present research was partially supported by the National Science Foundation (DMS 1550918)}
\begin{document}

\begin{abstract}
In this paper, we prove the moderate deviations principle (MDP) for a general system of slow-fast dynamics. We provide a unified approach, based on weak convergence ideas and stochastic control arguments, that cover both the averaging and the homogenization regimes. We allow the coefficients to be in the whole space and not just the torus and allow the noises driving the slow and fast processes to be correlated arbitrarily. Similar to the large deviation case, the methodology that we follow allows construction of provably efficient Monte Carlo methods for rare events that fall into the moderate deviations regime.
\end{abstract}

\maketitle

\section{Introduction}
The goal of this paper is to study moderate deviations for a large class of multiscale diffusion processes with small noise. In particular, we consider the system of slow-fast dynamics
\begin{align}
  dX_t^\varepsilon &= \left[ \frac{\varepsilon}{\delta} b(X_t^\varepsilon, Y_t^\varepsilon) + c(X_t^\varepsilon, Y_t^\varepsilon) \right] \,dt + \sqrt{\varepsilon} \sigma(X_t^\varepsilon, Y_t^\varepsilon) \,dW_t  \label{E:generalSDE}\\
  dY_t^\varepsilon &= \frac{1}{\delta} \left[ \frac{\varepsilon}{\delta} f(X_t^\varepsilon, Y_t^\varepsilon) + g(X_t^\varepsilon, Y_t^\varepsilon) \right] \,dt + \frac{\sqrt{\varepsilon}}{\delta} \left[ \tau_1(X_t^\varepsilon, Y_t^\varepsilon) \,dW_t + \tau_2(X_t^\varepsilon, Y_t^\varepsilon) \,dB_t \right]  \notag \\
  & X_0^\varepsilon = x_0, \quad Y_0^\varepsilon = y_0 \notag
\end{align}
for $t \in [0,1]$ such that $(X_t^\varepsilon, Y_t^\varepsilon) \in \mathbb{R}^n \times \mathbb{R}^d$. For convenience, we refer to the state space of $Y^\varepsilon$ as $\mathcal{Y}$. The parameter $\varepsilon \ll 1$ represents the strength of the noise while $\delta \ll 1$ is the time-scale separation parameter. $W_t$ and $B_t$ are independent $m$-dimensional Brownian motions.

In \eqref{E:generalSDE}, $X^\varepsilon$ is the slow motion and $Y^\varepsilon$ is the fast motion. Depending on the order in which $\varepsilon, \delta$ go to zero, we get different behavior, and in particular we are interested in the following regimes:
\begin{equation*}
  \lim_{\varepsilon \downarrow 0} \frac{\varepsilon}{\delta} = \begin{cases}
     \infty, & \text{Regime 1}, \\
     \gamma \in (0, \infty), & \text{Regime 2}.
   \end{cases}
\end{equation*}

The goal of moderate deviations is to study the behavior of the process ($X^{\varepsilon}$ in our case) in the regime between the central limit theorem behavior and the large deviation behavior. To be more precise, let $h(\varepsilon) \to + \infty$ such that $\sqrt{\varepsilon} h(\varepsilon) \to 0$ as $\varepsilon \downarrow 0$, denote by $\bar{X}_t = \lim_{\varepsilon \downarrow 0} X_t^\varepsilon$ (in the appropriate sense) the law of large numbers, and define the moderate deviation process
\begin{equation*} \label{E:deviations}
  \eta_t^\varepsilon = \frac{X_t^\varepsilon - \bar{X}_t}{\sqrt{\varepsilon} h(\varepsilon)}.
\end{equation*}
The goal is to derive the large deviation principle for $\eta_t^\varepsilon$, which is the moderate deviations principle for $X_{t}^{\varepsilon}$. Notice that if $h(\varepsilon) = 1$ then the limiting behavior of $\eta_t^\varepsilon$ is that of the central limit theorem (CLT) whereas if $h(\varepsilon) = 1 / \sqrt{\varepsilon}$ then we would get the large deviation result.

Both large and moderate deviations theory have a long history. For general results on large deviations, we refer the interested reader to classical manuscripts such as \cite{FW84 , DE97}. In regards to moderate deviations for diffusion processes, one of the first results was derived in \cite{BF77, F78} even though the analysis there was restricted to the setup with $b = \sigma = 0$ and under abstract conditions. In \cite{G03} the author studies the moderate deviations for \eqref{E:generalSDE} in the case of $\varepsilon = \delta$ with $b=0$ (averaging regime) and with the fast process $Y_t^\varepsilon$ being independent of the driving noise of the slow process $X_t^\varepsilon$, using different methods. In \cite{GL05} the authors study the MDP for integrated functionals of systems like \eqref{E:generalSDE}, in the averaging regime (i.e. when $b=0$) and with the fast process being independent of the driving noise of the slow process. In addition we also mention here the recent work of \cite{DJ15} where the moderate deviations principle (MDP) is derived for recursive stochastic algorithms (without multiple scales) using the weak convergence approach of \cite{DE97}.

We conclude this literature review by mentioning that the CLT for $X_t^\varepsilon$, i.e.\ when $h(\varepsilon) = 1$, has been derived in \cite{S14}. The LDP for $X_t^\varepsilon$ is studied in a series of papers \cite{DS12, S13, S15} for the cases of fast motion in periodic or in random stationary environments. Large deviations results for averaging problems have also been obtained in \cite{FS,Veretennikov2,CL10}.

The novelty of this paper is fourfold. First, we obtain an explicit form for the action functional of the MDP which is given in terms of solutions to auxiliary, but specific, Poisson equations that can be solved either analytically or numerically. This makes the computation of the action functional possible for a wide range of models, in contrast to existing literature where that was possible only for a more restrictive class of models (for instance we provide the MDP also in the homogenization regime, i.e., in Regime 1 with $b\neq 0$). We also illustrate this with a number of examples. Second, the method of our proof relies on the weak convergence approach of \cite{DE97} which allows us to connect the moderate deviations problem with a stochastic control problem. As in the case of large deviations (see \cite{S13, DSW12}), the solution to the stochastic control problem gives vital information for the design of efficient Monte Carlo methods for estimation of moderate deviations probabilities of interest. We plan to address the design of Monte Carlo methods based on the moderate deviations principle in a subsequent work. Third, we treat both the averaging regime, Regime 2 or Regime 1 with $b = 0$, and the homogenization regime, Regime 1 with $b \ne 0$, in a unified way. Fourth, the fast process $Y_t^\varepsilon$ is allowed to be both fully correlated with the slow process $X_t^\varepsilon$ and is also allowed to take values in the whole Euclidean space and not just on the torus. The latter fact complicates the mathematical analysis significantly and in particular, the proof of tightness. We gather all of the necessary technical results in Appendixes B and C.

The rest of the paper is organized as follows. In Section 2, we introduce notation and model conditions and state the main result as Theorem \ref{T:main}. In Section 3, we present examples of the MDP. In Sections 4, 5, and 6 we prove Theorem \ref{T:main}. In Section 4, we introduce the stochastic control representation, show the connection with the MDP, and define the concept of viable pairs, which is essential to the proof. In Section 5, we prove the MDP for Regime 1. In Section 6, we discuss the changes in the proof necessary for Regime 2. In the Appendix, we prove several auxiliary lemmas which are used in the main proof and constitute the main technical challenges.

\section{Notation, Conditions, and Main Results}

\subsection{Notation, Conditions, and Preliminaries}
We work with the canonical filtered probability space $(\Omega, \mathscr{F}, \mathbb{P})$ equipped with a filtration $\mathscr{F}_t$ that is right continuous and $\mathscr{F}_0$ contains all $\mathbb{P}$-negligible sets.

For given sets $A,B$, for $i,j\in\mathbb{N}$ and $\alpha\in(0,1)$ we denote by $\mathcal{C}_b^{i, j + \alpha}(A \times B)$, the space of functions with $i$ bounded derivatives in $x$ and $j$ derivatives in $y$, with all partial derivatives being $\alpha$-H\"{o}lder continuous with respect to $y$, uniformly in $x$.

We impose the following conditions on the SDE \eqref{E:generalSDE}.

\begin{condition} \label{C:growth}
\begin{enumerate}[(i)]
\item Let $\tilde{h}$ be either of the functions $b$ or $c$. $\tilde{h}(\cdot, y) \in \mathcal{C}^2(\mathbb{R}^n)$ for all $y \in \mathcal{Y}$, $\nabla_y \nabla_y \tilde{h} \in \mathcal{C}(\mathbb{R}^n\times\mathcal{Y})$, $\tilde{h}(x, \cdot) \in \mathcal{C}^\alpha(\mathcal{Y})$ uniformly in $x \in \mathbb{R}^n$ for some $\alpha \in (0, 1)$, and there exist $K$ and $q_{b},q_{c}\geq 0$ such that
\begin{align*}
    \lvert b(x, y) \rvert + \lVert \nabla_x b(x, y) \rVert + \lVert \nabla_x \nabla_x b(x, y) \rVert \le K (1 + \lvert y \rvert^{q_{b}})\nonumber\\
    \lvert c(x, y) \rvert + \lVert \nabla_x c(x, y) \rVert + \lVert \nabla_x \nabla_x c(x, y) \rVert \le K (1 + \lvert y \rvert^{q_{c}})\nonumber
\end{align*}
\item For every $N >0$ there exists a constant $C(N)$ such that for all $x_1, x_2 \in \mathbb{R}^n$ and $\lvert y \rvert \le N$, the diffusion matrix $\sigma$ satisfies
\begin{equation*}
    \lVert \sigma(x_1, y) - \sigma(x_2, y) \rVert \le C(N) \lvert x_1 - x_2 \rvert.
\end{equation*}
Moreover, there exists $K > 0$ and $q_{\sigma} \geq 0 $ such that
\begin{equation*}
    \lVert \sigma(x, y) \rVert \le K (1 + \lvert y \rvert^{q_{\sigma}}).
\end{equation*}
\item The functions $f(x, y)$, $g(x, y)$, $\tau_1(x, y)$, and $\tau_2(x, y)$ are $\mathcal{C}_b^{2, 2 + \alpha}(\mathbb{R}^n\times \mathcal{Y})$ with $\alpha \in (0, 1)$. %\\Namely, they have two bounded derivatives in $x$ and $y$, with all partial derivatives being $\alpha$-H\"{o}lder continuous with respect to $y$, uniformly in $x$.
In addition, $g$ is uniformly bounded.
\end{enumerate}
\end{condition}

\begin{condition} \label{C:ergodic}
\begin{enumerate}[(i)]
\item The diffusion matrix $\tau_1 \tau_1^\T + \tau_2 \tau_2^\T$ is uniformly continuous and bounded, nondegenerate and there exist constants $\beta_{1},\beta_{2}>0$ such that
\[
0<\beta_{1}\leq\frac{\left<(\tau_{1}\tau_{1}^{\T}(x,y)+\tau_{2}\tau_{2}^{\T}(x,y))y,y\right>}{|y|^{2}}\leq \beta_{2}.
\]
\item There exists $R,\Gamma>0$ and $r\geq 0$ such that in Regime 1,
\[
\sup_{x \in \mathbb{R}^n} f(x, y) \cdot y \leq  - \Gamma |y|^{r+1} \text{ for } |y|>R,
\]
and in  Regime 2,
\[
\sup_{x \in \mathbb{R}^n} (\gamma f(x, y) + g(x, y)) \cdot y \leq  - \Gamma |y|^{r+1} \text{ for } |y|>R.
\]
\end{enumerate}
\end{condition}

For each Regime $i = 1, 2$, define an operator $\opL_{i,x}$ (treating $x$ as a parameter) by
\begin{align}
  \opL_{1,x} F(y) &= \big( \nabla_y F(y) \big) f(x, y) + \frac{1}{2} (\tau_1 \tau_1^\T + \tau_2 \tau_2^\T) (x, y) : \nabla_y \nabla_y F(y),\label{Eq;Operators} \\
  \opL_{2,x} F(y) &= \big( \nabla_y F(y) \big) (\gamma f(x, y) + g(x,y)) + \gamma \frac{1}{2} (\tau_1 \tau_1^\T + \tau_2 \tau_2^\T) (x, y) : \nabla_y \nabla_y F(y)\nonumber
\end{align}
where the notation $A:B$ for two $n \times k$ matrices means the trace of their product,
\begin{equation*}
  A:B = \sum_{i=1}^n \sum_{j=1}^k a_{ij} b_{ij}.
\end{equation*}
For a $k \times k$ matrix $A$ and a $n$-dimensional vector--valued function of a $k$-dimensional vector $f(x)$ define $A : \nabla \nabla f$ as a $n$-dimensional vector where component $i$ is equal to $A : \nabla \nabla f_i$. Also, for notational convenience we sometimes collect the variables at the end of the expression and we write
\begin{equation*}
  \tau \tau^\T (x, y) = \tau(x, y) \tau(x, y)^\T.
\end{equation*}

Operators $\opL_{1,x}$ and $\opL_{2,x}$ are the infinitesimal generators for the processes that play the role of the fast motion (and with respect to which averaging is being performed) in Regimes 1 and 2 respectively.  Condition \ref{C:ergodic} guarantees that the fast process in each Regime $i = 1, 2$ has a unique invariant measure, denoted by $\mu_{i, x}(dy)$, for each $x\in\mathbb{R}^{n}$.

Because the fast motion takes values in an unbounded space, $\mathbb{R}^{d}$, the constants $q_{b},q_{c},q_{\sigma}$ that determine the growth of the coefficients from Condition \ref{C:growth} and the constant $r$ from Condition \ref{C:ergodic} that determines the recurrent properties of the fast component, will need to be related in order for the subsequent tightness argument to go through. In particular, we have Condition \ref{C:Tightness}.
\begin{condition}\label{C:Tightness}
Consider the constants $q_{b},q_{c},q_{\sigma}$ from Condition \ref{C:growth} and the constant $r$ from Condition \ref{C:ergodic}. Define $q_{b,c}=\max\{q_{b},q_{c}\}$ and $q_{F}=\max\{q_{b},q_{c},(q_{b}+1-r)^{+}\}$, where for any $x\in\mathbb{R}$ we have set $(x)^{+}=x 1_{x\geq 0}$. Then in Regime 1, we assume that
\begin{align*}
&\max\left\{(q_{F}+1-r)^{+}+q_{b,c}, (q_{F}+2(1-r))^{+}+q_{b,c},(q_{F}+1-r)^{+}+2q_{\sigma},\right.\nonumber\\
&\qquad\left.(q_{F}+2(1-r))^{+}+2q_{\sigma},(q_{F}+3(1-r))^{+}+2q_{\sigma}\right\}\leq r,\nonumber\\
&\max\left\{q_{F},q_{\sigma}, (q_{F}+1-r)^{+}\right\}<r.
\end{align*}
In Regime 2, we assume
\begin{align*}
&\max\left\{(q_{b,c}+1-r)^{+}+q_{b,c}, (q_{b,c}+2(1-r))^{+}+q_{b,c},(q_{b,c}+1-r)^{+}+2q_{\sigma},\right.\nonumber\\
&\qquad\left.(q_{b,c}+2(1-r))^{+}+2q_{\sigma},(q_{b,c}+3(1-r))^{+}+2q_{\sigma}\right\}\leq r,\nonumber\\
&\max\left\{q_{b,c},q_{\sigma}, (q_{b,c}+1-r)^{+}\right\}<r.
\end{align*}
\end{condition}

\begin{remark}Of course, it is clear that if the fast process is the Ornstein-Uhlenbeck process for example, where $r=1$, Condition \ref{C:Tightness} can be dramatically simplified, see Example 1 in Section \ref{S:Examples}. In addition, it is also clear that Condition \ref{C:Tightness} places some restrictions on $r$ as well. For example, if all the coefficients are bounded, in which case $q_{b}=q_{c}=q_{\sigma}=0$, then we need to have that $r\ge4/5$ for Regime 1 and $r\geq3/4$ for Regime 2.
\end{remark}

In addition, in Regime 1, we impose the following centering condition.
\begin{condition} \label{C:center}
The drift term $b$ satisfies
\begin{equation*}
    \int_\mathcal{Y} b(x, y) \mu_{1, x}(dy) = 0.
\end{equation*}
\end{condition}

Then by the results in \cite{PV01, PV03}, which we collected in Theorem \ref{T:regularity} in the Appendix, for each $\ell \in 1, \dots, n$, there is a unique, twice differentiable function $\chi_\ell(x, y)$ in the class of functions that grows at most polynomially in $\lvert y \rvert$ that satisfies the equation
\begin{equation} \label{E:cell}
  \opL_{1,x} \chi_\ell(x,y) = - b_\ell(x, y), \qquad \int_\mathcal{Y} \chi_\ell(x, y) \mu_{1,x}(dy) = 0, \text{ for }\ell=1,\cdots,n,
\end{equation}
where $b_{\ell}(x,y)$ is the $\ell^{\text{th}}$ component of the vector $b(x,y)=\left(b_{1}(x,y),\cdots, b_{n}(x,y)\right)$. Let us set $\chi(x,y) = (\chi_1(x,y), \dots, \chi_n(x, y))$. Define the function $\lambda_i(x, y) \colon \mathbb{R}^n \times \mathcal{Y} \to \mathbb{R}^n$ under Regime $i$ by
\begin{align*}
  \lambda_1(x, y) &= \big( \nabla_y \chi (x, y) \big) g(x, y) + c(x, y) \\
  \lambda_2(x, y) &= \gamma b(x, y) + c(x,y).
\end{align*}

Under Regime $i$, for any function $G(x,y)$, define the averaged function $\bar{G}$ by
\begin{equation} \label{E:Gbar}
  \bar{G}(x) = \int_{\mathcal{Y}} G(x, y) \mu_{i,x}(dy).
\end{equation}
It follows that $\bar{G}$ inherits the continuity and differentiability properties of $G$. In particular, for each regime,
\begin{equation*}
  \bar{\lambda}_i(x) = \int_{\mathcal{Y}} \lambda_i(x, y) \mu_{i,x}(dy).
\end{equation*}

Then by an argument similar to that of Theorem 3.2 in \cite{S13}, as $\varepsilon \downarrow 0$, in Regime $i$ we have the averaging result $X_t^\varepsilon \to \bar{X}_t$ in probability, where $\bar{X}_t$ is defined by
\begin{equation*} \label{E:averagedSDE}
  d\bar{X}_t = \bar{\lambda}_i(\bar{X}_t) \,dt, \qquad \bar{X}_0 = x_0.
\end{equation*}

Lastly, for Regime $i=1,2$, introduce the function $\Phi_i(x, y)$, given by the PDE
\begin{equation} \label{E:Phi}
  \opL_{i,x} \Phi_i(x, y) = - (\lambda_i(x,y) - \bar{\lambda}_i(x) ), \qquad \int_\mathcal{Y} \Phi_i(x, y) \mu_{i,x}(dy) = 0.
\end{equation}

Under our assumptions, each one of $\lambda_i-\bar{\lambda_{i}}$, for $i=1,2$, satisfy the assumptions of Theorem \ref{T:regularity}, part (iii), and thus
by Theorem \ref{T:regularity}, \eqref{E:Phi} has a unique classical solution in the class of functions which grow at most polynomially in $\lvert y \rvert$ for every $x$.

Last but not least we assume uniqueness of a strong solution.
\begin{condition}\label{C:StrongSolution}
We assume that the SDE \eqref{E:generalSDE} has a unique strong solution.
\end{condition}
\begin{remark}
Condition \ref{C:StrongSolution} holds for example if the coefficients are Lipshcitz continuous with at most linear growth. However, these conditions can be significantly weakened, see for example \cite{Verettenikov1}. Conditions \ref{C:growth},  \ref{C:Tightness} and \ref{C:StrongSolution} should be considered together and it is clear that depending on the value of $r$ in the recurrence Condition \ref{C:ergodic}, Conditions \ref{C:growth},  \ref{C:Tightness} will directly imply Condition \ref{C:StrongSolution}. For example if $r=1$ then the coefficients cannot grow faster than linearly in $y$ and are always assumed to be bounded in $x$, so in that case for instance Condition \ref{C:StrongSolution} instantly holds.
\end{remark}

\subsection{Main Results}

By \cite{DE97}, the LDP for $\eta_t^\varepsilon$ is equivalent to the Laplace principle, which states that for any bounded continuous function $a \colon \mathcal{C}([0, 1]; \mathbb{R}^n) \to \mathbb{R}$,
\begin{equation} \label{E:LaplacePrinciple}
  \lim_{\varepsilon \downarrow 0} - \frac{1}{h^2(\varepsilon)} \log \E \left[ \exp\left\{- h^2(\varepsilon) a(\eta^\varepsilon) \right\} \right] = \inf_{\xi \in C([0, 1]; \mathbb{R}^n)} \left( S(\xi) + a(\xi) \right)
\end{equation}
where $S(\xi)$ is called the action functional. In this paper we essentially prove \eqref{E:LaplacePrinciple} and Theorem \ref{T:main} identifies the action functional $S(\xi)$. In order to state Theorem \ref{T:main}, we need to know the relative rates at which $\delta$, $\varepsilon$, and $1 / h(\varepsilon)$ vanish. In particular, in Regime $i$, $i = 1, 2$, define $j_1$, $j_2$ by
\begin{equation}\label{Eq:LimitingConstants}
  j_1 = \lim_{\varepsilon \downarrow 0} \frac{\delta / \varepsilon}{\sqrt{\varepsilon} h(\varepsilon)}<\infty, \qquad
  j_2 = \lim_{\varepsilon \downarrow 0} \frac{\varepsilon / \delta - \gamma}{\sqrt{\varepsilon} h(\varepsilon) }<\infty.
\end{equation}
$j_1$, $j_2$ specifies the relative rate at which $\varepsilon/\delta$ goes to its limit and $h(\varepsilon)$  goes to infinity. In order for a moderate deviations principle to hold, we require that $j_1$, $j_2$ be finite.

The main result of this paper is the following theorem.
\begin{theorem} \label{T:main}
Let Conditions \ref{C:growth}, \ref{C:ergodic}, \ref{C:Tightness} and \ref{C:StrongSolution} be satisfied. Additionally, under Regime 1, let Condition \ref{C:center} be satisfied. Then under Regime $i$, $i = 1,2$, the process $\{X^{\varepsilon},\varepsilon>0\}$ from \eqref{E:generalSDE} satisfies the MDP, with the action functional $S(\xi)$ given by
\begin{equation*}
  S(\xi) = \frac{1}{2} \int_0^1 \left(\dot{\xi}_s - \kappa \left(\bar{X}_s, \xi_{s} \right) \right)^\T q^{-1}(\bar{X}_s) \left(\dot{\xi}_s - \kappa \left(\bar{X}_s, \xi_{s} \right) \right) ds
\end{equation*}
if $\xi \in \mathcal{C}([0, 1]; \mathbb{R}^n)$ is absolutely continuous, and $\infty$ otherwise.
Under Regime 1, we have
\begin{align}
 \label{E:rDef} \kappa(x, \eta) &= \big( \nabla_x \bar{\lambda}_1(x) \big) \eta + j_1 \int_\mathcal{Y} \big( \nabla_y \Phi_1(x,y) \big) g(x, y) \mu_{1,x}(dy) \\
 \label{E:qDef} q(x) &= \int_\mathcal{Y} \left( \alpha_1 \alpha_1^\T(x, y) + \alpha_2 \alpha_2^\T(x, y) \right) \mu_{1,x}(dy) \\
 \label{E:alphaDef} \alpha_1(x,y) &= \sigma(x,y) + \big( \nabla_y \chi(x, y) \big) \tau_1(x, y), \
  \alpha_2(x,y) = \big( \nabla_y \chi(x, y) \big) \tau_2(x, y).
\end{align}

Under Regime 2, we have
\begin{align*}
  &\kappa(x, \eta) =\big( \nabla_x \bar{\lambda}_2 (x) \big) \eta+j_2 \int_\mathcal{Y} \left[ b(x, y)  -\frac{1}{\gamma} \big( \nabla_y \Phi_2(x, y) \big)  g(x, y)  \right] \mu_{2, x} (dy)\\
  &q(x) = \int_\mathcal{Y} \left( \alpha_1 \alpha_1^\T(x, y) + \alpha_2 \alpha_2^\T(x, y) \right) \mu_{2, x}(dy) \\
    &\alpha_1(x,y) = \sigma(x,y) + \big( \nabla_y \Phi_2(x, y) \big) \tau_1(x, y), \quad
  \alpha_2(x,y) = \big( \nabla_y \Phi_2(x, y) \big) \tau_2(x, y),
\end{align*}
where the finite constants $j_{1}$, $j_{2}$ are defined in (\ref{Eq:LimitingConstants}).
\end{theorem}

\begin{remark} Note that in either Regime, the function $\kappa(x, \eta)$ is affine in $\eta$ and the function $q(x)$ is constant in $\eta$. This is expected by the nature of moderate deviations. In the large deviations case, see \cite{DS12, S13}, the corresponding $\kappa(x)$ and $q(x)$ are nonlinear functions of $x$. The affine structure of $\kappa(x, \eta)$ is what makes the moderate deviations very appealing for the design of Monte Carlo simulation methods, as it makes the solution to the associated Hamilton-Jacobi-Bellman equation much easier to obtain. We plan to explore this in detail in a follow up work.
\end{remark}

\section{Examples}\label{S:Examples}

In this section we present some concrete examples to illustrate Theorem \ref{T:main}.

\subsection{Example 1}
Consider the system of one-dimensional processes
\begin{align*}
  dX_t^\varepsilon &= b(X_t^\varepsilon, Y_{t}^\varepsilon) \,dt + \sqrt{\varepsilon} \sigma(X_t^\varepsilon, Y_{t}^\varepsilon) \,dW_t, \qquad X_0^\varepsilon = x_0, \\
  dY_t^\varepsilon &= - \frac{1}{\varepsilon} \frac{1}{2} Y_t^\varepsilon \,dt + \frac{1}{\sqrt{\varepsilon}} dB_t \notag
\end{align*}
where $B$ and $W$ are independent Brownian motions. The invariant measure of the fast process $Y$ is the Gaussian measure given by $\mu_{2,x}(dy) = (2\pi)^{-1/2} \exp(-y^2 / 2) \,dy$. This system can be rewritten in terms of \eqref{E:generalSDE} with $\delta = \varepsilon$. In this case, the recurrence constant $r$ from  Condition \ref{C:ergodic} is $r=1$ and the restrictions on $q_{b}$, $q_{\sigma}$ from Condition \ref{C:Tightness} take the much simpler form $q_{b}\leq 1/2$ and $2q_{\sigma}+q_{b}\leq 1$. Notice that the limit $\bar{X}_t = \lim_{\varepsilon \downarrow 0} X_t^\varepsilon$ is given by
\begin{equation*}
  d \bar{X}_t = \bar{\lambda}_2(\bar{X}_t) \,dt, \ \bar{X}_0 = x_0, \text{ where }
  \bar{\lambda}_2(x) = \frac{1}{\sqrt{2\pi}} \int_\mathbb{R} b(x, y) e^{-y^2 / 2} \,dy .
\end{equation*}
In this case $\Phi_{2}(x, y)$, i.e.\ the solution to the PDE \eqref{E:Phi}, takes the explicit form
\begin{equation*}
  \frac{\partial \Phi_2}{\partial y}(x, y) = -2 e^{y^2/2} \int_{-\infty}^y b(x, z) e^{-z^2/2} \,dz
\end{equation*}
which then implies that the action functional $S(\xi)$ of Theorem \ref{T:main} is defined with
\begin{align*}
  \kappa(x, \eta) &= \frac{1}{\sqrt{2\pi}} \eta \int_\mathbb{R} \frac{\partial b}{\partial x}(x, y) e^{-y^2 / 2} \,dy \\
  q(x) &= \int_\mathbb{R} \left[ \sigma(x, y)^2 + 4 e^{y^2} \left(  \int_{-\infty}^y b(x, z) e^{-z^2/2} \,dz \right)^2 \right] \,\mu_{2,x}(dy).
\end{align*}

\begin{remark}
\cite{G03} presents a similar example under the additional assumption that $\int_\mathbb{R} b(x_0, y) \mu_{2,x}(dy) = 0$. By Theorem 1, this assumption is not necessary, and the results here extend the results of \cite{G03} to a much more general class of processes in a unified way.
\end{remark}

\subsection{Example 2}
In the second example, we consider the first order Langevin equation under Regime 1,
\begin{equation*}
  dX_t^\varepsilon = \left[ - \frac{\varepsilon}{\delta} \nabla Q \left( \frac{X_t^\varepsilon}{\delta} \right) - \nabla V (X_t^\varepsilon) \right] \,dt + \sqrt{\varepsilon} \sqrt{2D} \,dW_t, \quad X_0^\varepsilon = x_0.
\end{equation*}
This equation has a number of applications and has been studied extensively, beginning with \cite{Z88}, see also \cite{DSW12}. In our notation, let $Y_t^\varepsilon = X_t^\varepsilon / \delta$, $b(x, y) = f(x, y) = - \nabla Q(y)$, and $c(x, y) = g(x, y) = - \nabla V(x)$. The invariant density $\mu(y)$ is the Gibbs measure
\begin{equation*}
  \mu(y) = \frac{1}{Z} e^{-Q(y) / D}, \quad Z = \int_\mathcal{Y} e^{-Q(y) / D} \,dy.
\end{equation*}
In order to have closed form formulas, let us also assume that $Q(y_1, y_2, \dots, y_d) = Q_1(y_1) + Q_2(y_2) + \dots + Q_d(y_d)$ and that $\mathcal{Y}$ is the $d$-dimensional unit torus. Since the fast motion is restricted to be on a torus, the recurrence condition (part (ii)) of Condition \ref{C:ergodic} and Condition \ref{C:Tightness} are not needed.

Then $\bar{X}_t = \lim_{\varepsilon \downarrow 0} X_t^\varepsilon$ is given by
\begin{align*}
  \bar{X}_t &= x_0 + \int_0^t \bar{\lambda}_1(\bar{X}_s) \,ds \\
\intertext{where}
  \bar{\lambda}_1(x) &= - \bar{\Theta} \nabla V(x), \quad \bar{\Theta} = \text{diag} \left[ \frac{1}{Z_1 \hat{Z}_1}, \dots , \frac{1}{Z_d \hat{Z}_d} \right] \\
    \intertext{and for $i = 1, 2, \dots, d$}
    Z_i &= \int_\mathbb{T} e^{-Q_i(y_i) / D} \,dy_i, \quad \hat{Z}_i = \int_\mathbb{T} e^{Q_i(y_i) / D} \,dy_i.
\end{align*}

$\Phi_1(x, y)$ is given by
\begin{align*}
  \nabla_y \Phi(x, y) &= \frac{1}{D} \Theta(y) \nabla V(x) \\
  \intertext{where}
  \Theta(y) &= \text{diag} \left[ \frac{e^{Q_i (y_i) / D}}{\hat{Z}_i} \left( y_i - \frac{1}{Z_i} \int_0^{y_i} e^{-Q_i(\xi)/D} \,d\xi \right. \right. \\
  &\left. \left .+ \frac{1}{\hat{Z}_i} \int_0^1 e^{Q_i(\rho)/D} \int_0^\rho \left( \frac{1}{Z_i} e^{-Q_i(\xi)/D} - 1 \right) \,d\xi \,d\rho \right) \right] .
\end{align*}

Then the action functional $S(\xi)$ of Theorem \ref{T:main} is defined with
\begin{align*}
  \kappa(x, \eta) &= -\bar{\Theta} \nabla \nabla V(x) \eta - j_1 \int_\mathcal{Y} \left( \frac{1}{D} \Theta(y) \nabla V(x) \right) \nabla V(x) \,\mu(dy) \\
  q(x) &= 2D \bar{\Theta}.
\end{align*}

\begin{remark} We remark that this example is not covered by previous results in the literature on moderate deviations. Here we are able to get a very explicit form for the action functional.
\end{remark}

\subsection{Example 3}
To illustrate the case where $Y_t^\varepsilon$ is a CIR (square-root) process, consider the following model:
\begin{align*}
dX_t^\varepsilon &= c(X_t^\varepsilon, Y_t^\varepsilon) \,dt + \sqrt{\varepsilon} \sigma(X_t^\varepsilon, Y_t^\varepsilon) \,dW_t, && X_0^\varepsilon = x_0 \in \mathbb{R}\\
dY_t^\varepsilon &= \frac{\varepsilon}{\delta^2} a(b - Y_t^\varepsilon) \,dt + \frac{\sqrt{\varepsilon}}{\delta} \tau \sqrt{Y_t^\varepsilon} \,dW_t, && Y_0^\varepsilon = y_0 \in \mathbb{R}
\end{align*}
where $a$, $b$, and $\tau$ are positive constants satisfying $2ab \ge \tau^2$. Note that this model does not satisfy Condition \ref{C:ergodic} because the fast process noise is degenerate at $y = 0$. However, if $y_0 > 0$ and $2ab \ge \tau^2$ then $Y_t^\varepsilon > 0$ for all $t > 0$ w.p.1., $Y_t^\varepsilon$ has the gamma distribution as its unique invariant measure and so the results are expected to hold. For this model, the invariant measure and the limiting process $\bar{X}_t$ do not depend on the regime. However, as we shall see the MDP for Regimes 1 and 2 do differ. The fast process has the gamma invariant density
\begin{equation*}
  m(y) = \frac{(2a/\tau^2)^{2ab/\tau^2}}{\Gamma(2ab/\tau^2)} y^{2ab/\tau^2 - 1} e^{-2ay/ \tau^2}.
\end{equation*}

Then $\bar{X}_t = \lim_{\varepsilon \downarrow 0} X_t^\varepsilon$ satisfies the ordinary differential equation
\begin{align*}
  \bar{X}_t &= x_0 + \int_0^t \bar{\lambda}(\bar{X}_s) \,ds \\
  \intertext{where}
  \bar{\lambda}(x) &= \int_0^\infty c(x, y) m(y) \,dy.
\end{align*}

Under Regime 1, the action functional $S(\xi)$ of Theorem \ref{T:main} is expected to be defined with
\begin{align*}
  \kappa(x, \eta) &= \eta \frac{d}{dx} \bar{\lambda}(x) \\
  q(x) &= \int_0^\infty \sigma^2(x,y) m(y) \,dy.
\end{align*}

In contrast, under Regime 2, we let $\Phi_{2}(x,y)$ be the unique solution to (\ref{E:Phi}) with $i=2$ and $\lambda_{2}(x,y)=c(x,y)$. Then we have that
\begin{align*}
  \kappa(x, \eta) &= \eta \frac{d}{dx} \bar{\lambda}(x) \\
  q(x) &= \int_0^\infty \left( \sigma(x, y) +  \tau \sqrt{y}\frac{d \Phi_2}{dy}(x, y) \right)^2 m(y) \,dy.
\end{align*}

Hence, the two MDP's differ on the formula for $q(x)$. Again, we remark that this example can be covered with the results of this paper, but it is not clear whether existing previous results in the literature can address it or indicate how the action functional should look like.

\section{The controlled processes} \label{S:controlledProcess}

The proof of the Laplace principle \eqref{E:LaplacePrinciple} is based on a stochastic control representation given by Theorem 3.1 in \cite{BD98}. This theorem is restated here for the convenience of the reader.
\begin{theorem} \label{T:Boue}
Let $Z$ be a $2m$-dimensional Brownian motion with respect to the filtration $\{\mathscr{F}_t\}$ for $0 \le t \le 1$. Let $\mathcal{A}$ be the space of $\mathscr{F}_t$-progressively measurable $2m$-dimensional processes $v = (v_1, v_2)$ for $0 \le t \le 1$ satisfying
\begin{equation*}
  \E \int_0^1 \lvert v(s) \rvert^2 \,ds < \infty.
\end{equation*}
Let $F$ be a bounded, measurable, real--valued function defined on the space of $\mathbb{R}^{2m}$--valued continuous functions on $[0,1]$. Then
\begin{equation*}
  - \log \E \left[ \exp \{ - F(Z(\cdot))\} \right] = \inf_{v \in \mathcal{A}} \E \left[ \frac{1}{2} \int_0^1 \lvert v(s) \rvert^2 \,ds + F\left( Z(\cdot) + \int_0^\cdot v(s) \,ds \right) \right].
\end{equation*}
\end{theorem}

In our case we set $Z(\cdot) = (W(\cdot), B(\cdot))$ and each one of $v_{1},v_{2}$ are $m-$dimensional vectors. Under Condition \ref{C:StrongSolution}, for each $\varepsilon > 0$, \eqref{E:generalSDE} has a unique strong solution. Therefore $\eta^\varepsilon$ is a measurable function of $Z$. Set $F(Z(\cdot)) = h^2(\varepsilon) a(\eta^\varepsilon(\cdot))$. Set $u_{i}^\varepsilon = v_{i}/h(\varepsilon)$, $u^{\varepsilon}=(u_{1}^\varepsilon,u_{2}^\varepsilon)$, and then divide by $h^2(\varepsilon)$ to obtain
\begin{equation} \label{E:controlRepresentation}
  - \frac{1}{h^2(\varepsilon)} \log \E \left[ \exp\{-h^2(\varepsilon)a(\eta^\varepsilon)\}\right] = \inf_{u^\varepsilon \in \mathcal{A}} \E \left[ \frac{1}{2} \int_0^1 \lvert u^\varepsilon(s) \rvert^2 \,ds + a(\eta^{\varepsilon,u^\varepsilon}) \right]
\end{equation}
where the controlled deviations process $\eta^{\varepsilon,u^\varepsilon}$ is defined by
\begin{equation} \label{E:controlledDeviations}
  \eta_t^{\varepsilon, u^\varepsilon} = \frac{1}{\sqrt{\varepsilon} h(\varepsilon)} \left( X_t^{\varepsilon, u^\varepsilon} - \bar{X}_t \right)
\end{equation}
and the controlled processes $X_t^{\varepsilon, u^\varepsilon}$ and $Y_t^{\varepsilon, u^\varepsilon}$ are defined by
\begin{align} \label{E:controlledSDE}
  dX_t^{\varepsilon, u^\varepsilon} &=  \left[ \frac{\varepsilon}{\delta} b(X_t^{\varepsilon, u^\varepsilon}, Y_t^{\varepsilon, u^\varepsilon}) + c(X_t^{\varepsilon, u^\varepsilon}, Y_t^{\varepsilon, u^\varepsilon}) +  \sqrt{\varepsilon} h(\varepsilon) \sigma(X_t^{\varepsilon, u^\varepsilon}, Y_t^{\varepsilon, u^\varepsilon}) u_1^\varepsilon(t) \right]  \,dt \\
  &\quad + \sqrt{\varepsilon} \sigma(X_t^{\varepsilon, u^\varepsilon}, Y_t^{\varepsilon, u^\varepsilon}) \,dW_t  \notag \\
  dY_t^{\varepsilon, u^\varepsilon} &= \frac{1}{\delta} \left[ \frac{\varepsilon}{\delta} f(X_t^{\varepsilon, u^\varepsilon}, Y_t^{\varepsilon, u^\varepsilon}) + g(X_t^{\varepsilon, u^\varepsilon}, Y_t^{\varepsilon, u^\varepsilon})  +  \sqrt{\varepsilon} h(\varepsilon) \tau_1(X_t^{\varepsilon, u^\varepsilon}, Y_t^{\varepsilon, u^\varepsilon}) u_1^\varepsilon(t) \right. \notag \\
  &\quad + \left. \sqrt{\varepsilon} h(\varepsilon) \tau_2(X_t^{\varepsilon, u^\varepsilon}, Y_t^{\varepsilon, u^\varepsilon}) u_2^\varepsilon(t) \right] \,dt \notag \\
  &\quad + \frac{\sqrt{\varepsilon}}{\delta} \left[ \tau_1(X_t^{\varepsilon, u^\varepsilon}, Y_t^{\varepsilon, u^\varepsilon}) \,dW_t + \tau_2(X_t^{\varepsilon, u^\varepsilon}, Y_t^{\varepsilon, u^\varepsilon}) \,dB_t \right]  \notag \\
  & X_0^{\varepsilon, u^\varepsilon} = x_0, \quad Y_0^{\varepsilon, u^\varepsilon} = y_0. \notag
\end{align}

Note that we can rewrite $\eta^{\varepsilon, u^\varepsilon}$ in the form
\begin{align}\label{E:deviationsIntegral}
  \eta_t^{\varepsilon, u^\varepsilon} &=   \int_0^t \frac{1}{\sqrt{\varepsilon} h(\varepsilon)} \left[ \frac{\varepsilon}{\delta} b(X_s^{\varepsilon, u^\varepsilon}, Y_s^{\varepsilon, u^\varepsilon}) + c(X_s^{\varepsilon, u^\varepsilon}, Y_s^{\varepsilon, u^\varepsilon}) - \bar{\lambda}_i(\bar{X}_s)\right] \,ds \\
  &\qquad + \int_0^t \sigma(X_s^{\varepsilon, u^\varepsilon}, Y_s^{\varepsilon, u^\varepsilon}) u_1^\varepsilon(s)\, ds + \int_0^t \frac{1}{h(\varepsilon)} \sigma(X_s^{\varepsilon, u^\varepsilon}, Y_s^{\varepsilon, u^\varepsilon}) \,dW_s. \notag
\end{align}

Define $\mathcal{Z} = \mathbb{R}^m$. This is the space in which the control processes $u_1^\varepsilon$ and $u_2^\varepsilon$ take values. Define $\theta_i(x, \eta, y, z_1, z_2) \colon \mathbb{R}^n \times \mathbb{R}^n \times \mathcal{Y} \times \mathcal{Z} \times \mathcal{Z} \to \mathbb{R}^n$ by
\begin{align} \label{E:theta}
& \theta_1(x, \eta, y, z_1, z_2) = \big( \nabla_y \chi(x, y) \big) (\tau_1(x, y) z_1 + \tau_2(x, y) z_2) + j_1 \big( \nabla_y \Phi_1(x, y) \big) g(x, y) \\
  &\hspace{3cm} + \big( \nabla_x \bar{\lambda}_1(x) \big) \eta + \sigma(x, y) z_1
 \notag \\
&  \theta_2(x, \eta, y, z_1, z_2) = j_2 b(x, y) + \big( \nabla_y \Phi_2(x, y) \big) [\tau_1(x, y) z_1 + \tau_2(x, y) z_2 ] + \big( \nabla_x \bar{\lambda}_2(x) \big) \eta \notag \\
  \notag &\quad+ \sigma(x, y) z_1 + j_2 \big( \nabla_y \Phi_2(x, y) \big) f(x, y) + \frac{j_2}{2} \left( \left( \tau_1 \tau_1^\T + \tau_2 \tau_2^\T \right) (x, y) : \nabla_y \nabla_y \Phi_2(x, y) \right)
\end{align}

Conditions \ref{C:growth}, \ref{C:Tightness} and Theorem \ref{T:regularity} guarantee  that the functions $\theta_{1}$ and $\theta_{2}$ are bounded in $x$, affine in $\eta,z_{1}$ and $z_{2}$ and bounded polynomially in $|y|$ with order $r\geq 0$ ($r$ comes from Condition \ref{C:ergodic}).

Next we introduce the occupation measure $P^{\varepsilon, \Delta}$. Let $\Delta = \Delta(\varepsilon) \downarrow 0$ as $\varepsilon \downarrow 0$, whose role is to exploit a time-scale separation. Let $A_1$, $A_2$, $B$, and $\Gamma$ be Borel sets of $\mathcal{Z} = \mathbb{R}^m$, $\mathcal{Z}$, $\mathcal{Y} = \mathbb{R}^d$, and $[0,1]$ respectively. Let $(X^{\varepsilon, u^\varepsilon}, Y^{\varepsilon, u^\varepsilon})$ solve \eqref{E:controlledSDE}. Associate with $(X^{\varepsilon, u^\varepsilon}, Y^{\varepsilon, u^\varepsilon})$ and $u^\varepsilon$ a family of occupation measures $P^{\varepsilon, \Delta}$ defined by
\begin{equation*} \label{E:occupation}
  P^{\varepsilon, \Delta}(A_1 \times A_2 \times B \times \Gamma) = \int_\Gamma \left[ \frac{1}{\Delta} \int_t^{t+\Delta} 1_{A_1}(u_1^\varepsilon(s)) 1_{A_2}(u_2^\varepsilon(s)) 1_B(Y_s^{\varepsilon, u^\varepsilon}) \,ds \right] \,dt
\end{equation*}
and assume $u_i^\varepsilon(s) = 0$ if $s > 1$.

\begin{definition} \label{D:ViablePair}
Let $\theta(x, \eta, y, z_1, z_2) \colon \mathbb{R}^n \times \mathbb{R}^n \times \mathcal{Y} \times \mathcal{Z} \times \mathcal{Z} \to \mathbb{R}^n$ be a function that has at most polynomial growth in $|y|$ with order $r\geq 0$. For each $x\in\mathbb{R}^{n}$, let $\opL_{x}$ be a second order elliptic partial differential operator and denote by $\mathcal{D}(\opL_{x})$ its domain of definition. A pair $(\psi, P) \in \mathcal{C}([0,1]; \mathbb{R}^n) \times \mathcal{P}(\mathcal{Z} \times \mathcal{Z} \times \mathcal{Y} \times [0,1])$ is called a viable pair with respect to $(\theta, \opL_{x})$ if
\begin{itemize}
\item The function $\psi$ is absolutely continuous.
\item The measure $P$ is integrable in the sense that
\begin{equation*}
  \int_{\mathcal{Z} \times \mathcal{Z} \times \mathcal{Y} \times [0,1]} \left[ \lvert z_1 \rvert^2 + \lvert z_2 \rvert^2 + \lvert y \rvert^{2r}\right] P(dz_1 \,dz_2 \,dy \,ds) < \infty.
\end{equation*}
\item For all $t \in [0,1]$,
\begin{equation}\label{E:viableSDE}
  \psi_t = \int_{\mathcal{Z} \times \mathcal{Z} \times \mathcal{Y} \times [0,t]} \theta(\bar{X}_s, \psi_s, y, z_1, z_2) \,P(dz_1 \,dz_2 \,dy \,ds).
\end{equation}
\item For all $t \in [0, 1]$ and for every $F \in \mathcal{D}(\opL_{x})$,
\begin{equation}\label{E:viableCenter}
  \int_0^t \int_{\mathcal{Z} \times \mathcal{Z} \times \mathcal{Y}} \opL_{\bar{X}_s} F(y) \,P(dz_1 \,dz_2 \,dy \,ds) = 0.
\end{equation}
\item For all $t \in [0, 1]$,
\begin{equation}\label{E:viableLebesgue}
  P(\mathcal{Z} \times \mathcal{Z} \times \mathcal{Y} \times [0,t]) = t.
\end{equation}
\end{itemize}
We  write $(\psi, P)\in\mathcal{V}(\theta, \opL_{x})$.
\end{definition}

Note that the last item is equivalent to stating that the last marginal of $P$ is Lebesgue measure, or that $P$ can be decomposed as $P(dz_1 \,dz_2 \,dy \,dt) = P_t(dz_1\, dz_2\, dy) \,dt$.  In comparison to the definition of viable pairs in the large deviations case (for example, \cite{DS12}), $\psi$ does not appear in \eqref{E:viableCenter}, and so $\psi$ and $P$ are decoupled. Another difference with the definition of viable pair in \cite{DS12} is that here we need to impose the  condition
$\int_{\mathcal{Z} \times \mathcal{Z} \times \mathcal{Y} \times [0,1]}  \lvert y \rvert^{2r} P(dz_1 \,dz_2 \,dy \,ds) < \infty$ which is due to the polynomial growth in $|y|$ of the involved functions. As we will see in the convergence proof, due to the a priori bound of Lemma \ref{L:Ygrowth}, this is a restriction that is satisfied.

The controlled process \eqref{E:controlledDeviations} and definition of viable pairs will be used to prove the following theorem:

\begin{theorem} \label{T:control}
  Let Conditions \ref{C:growth}, \ref{C:ergodic}, \ref{C:Tightness} and \ref{C:StrongSolution} be satisfied. Additionally, under Regime 1, let Condition \ref{C:center} be satisfied. Then under Regime $i$, $i = 1,2$, the family of processes $\{X^{\varepsilon},\varepsilon>0\}$ from \eqref{E:generalSDE} satisfies the MDP, with the action functional $S(\xi)=S_{i}(\xi)$ given by
\begin{equation*}
  S_{i}(\xi) = \inf_{(\xi,P) \in \mathcal{V}(\theta_i, \opL_{i,x})} \left[ \frac{1}{2} \int_{\mathcal{Z} \times \mathcal{Z} \times \mathcal{Y} \times [0, 1]} \left[ \lvert z_1 \rvert^2 + \lvert z_2 \rvert^2 \right] P(dz_1 \, dz_2 \, dy \, ds) \right]
\end{equation*}
with the convention that the infimum over the empty set is $\infty$.
\end{theorem}

Notice that Theorem \ref{T:control} offers a compact way to write the MDP for both regimes in terms of the appropriate viable pairs each time. As will be shown during the proof, Theorem \ref{T:main} follows directly from Theorem \ref{T:control}.

\section{Proof in Regime 1}

The proof is nearly identical for Regime 1 and for Regime 2, aside from some technical differences. In this section, we present the proof for Regime 1. In Section 6, we discuss the changes necessary for Regime 2. In Subsections \ref{S:Tightness} and \ref{SS:ExistenceViablePair} we prove tightness and convergence of the pair $(\eta^{\varepsilon, u^\varepsilon}, P^{\varepsilon, \Delta})$ respectively. In Subsection 5.3, we prove the Laplace principle lower bound. In Subsection 5.4, we prove compactness of level sets of $S(\cdot)$. Finally, in Subsection 5.5, we prove the Laplace principle upper bound and the representation formula of Theorem \ref{T:main}.

\subsection{Proof of tightness}\label{S:Tightness}
The main result of this section is the following proposition on tightness.
\begin{proposition}
\label{P:tightness}Let Conditions \ref{C:growth}, \ref{C:ergodic}, \ref{C:Tightness}, \ref{C:center} and \ref{C:StrongSolution} be satisfied.  Consider any family $\{u^{\varepsilon},\varepsilon>0\}$ of
controls in $\mathcal{A}$ satisfying for some $N<\infty$
\begin{equation}
\sup_{\epsilon>0}\int_{0}^{1}\left| u^{\varepsilon}(t)\right|
^{2}dt<N, \text{almost surely}\label{A:UniformlyAdmissibleControls}%
\end{equation}
Then the following hold.

\begin{enumerate}
\item {The family $\{(X^{\varepsilon,u^{\varepsilon}},\mathrm{P}^{\varepsilon,\Delta
}),\varepsilon>0\}$ is tight.}

\item {Define the set
\[
\mathcal{B}_{r,M}=\left\{(z_{1},z_{2},y)\in\mathcal{Z}\times\mathcal{Z}\times \mathcal{Y}:\left(|z_{1}|>M, |z_{2}| > M, |y|^{r}>M\right)\right\}.
\]

The family $\{\mathrm{P}^{\varepsilon,\Delta},\varepsilon>0\}$ is uniformly
integrable in the sense that
\[
\lim_{M\rightarrow\infty}\sup_{\epsilon>0}\mathbb{E}_{x_{0},y_{0}}\left[
\int_{\{(z_{1},z_{2},y)\in
\mathcal{B}_{r,M}\times
\lbrack0,1]}\left[\left| z_{1}\right| +\left| z_{2}\right|+\left| y\right|^{r} \right]\mathrm{P}^{\epsilon,\Delta}%
(dz_{1} \,dz_{2} \,dy \,dt)\right]  =0.
\]
}
\end{enumerate}
\end{proposition}

The proof of Proposition \ref{P:tightness} is the subject of Sections \ref{SS:tightnessP} and \ref{SS:tightnessProcess}.
\subsubsection{Tightness of $\{ P^{\varepsilon, \Delta}, \varepsilon,\Delta>0 \}$ on $\mathcal{P}(\mathcal{Z} \times \mathcal{Z} \times \mathcal{Y} \times [0,1])$ } \label{SS:tightnessP}

The argument for tightness is similar to the argument for tightness in the proof of Theorem 3.2 in \cite{S13} (see also \cite{DS12}), but with some differences due to the unboundedness of the space on which the fast motion takes values. We repeat here for completeness the argument emphasizing the differences.

By Lemmas \ref{L:uBound} and \ref{L:Ygrowth} in the Appendix, we can restrict to a family $\{ u^\varepsilon = (u_1^\varepsilon, u_2^\varepsilon), \varepsilon > 0 \}$ of controls in $\mathcal{A}$ satisfying
\begin{equation*}
  \sup_{\varepsilon > 0} \E \int_0^1 \left[ \lvert u_1^\varepsilon(s) \rvert^2 + \lvert u_2^\varepsilon(s) \rvert^2 + \lvert Y^{\varepsilon,u^{\varepsilon}}_{s} \rvert^{2r}\right] \,ds < \infty.
\end{equation*}

Recall that a tightness function $\hat{g}(x)$ is a function mapping a space $\mathcal{X}$ to $\mathbb{R} \cup \{ \infty \}$ which has a lower bound and for which for each $M < \infty$, the level set $Z_{\hat{g}}(M) = \{ x \in \mathcal{X} \colon \hat{g}(x) \le M \}$ is relatively compact in $\mathcal{X}$.

Consider $q\in \mathcal{P}(\mathcal{Z} \times \mathcal{Z} \times \mathcal{Y} \times [0,1])$ (not to be confused with the growth parameters of Condition \ref{C:growth}). The function \[
\hat{g}(q) = \int_{\mathcal{Z} \times \mathcal{Z} \times \mathcal{Y} \times [0,1]} \left[ \lvert z_1 \rvert^2 + \lvert z_2 \rvert^2 + \lvert y \rvert^{2r}\right] q(dz_1\,dz_2\,dy\,dt)\]
is a tightness function on $\mathcal{P}(\mathcal{Z} \times \mathcal{Z} \times \mathcal{Y} \times [0,1])$ by the facts that it is nonnegative and that the level sets of $\hat{g}$ are relatively compact. Then by Theorem A.3.17 in \cite{DE97}, for each $M < \infty$, the set
\begin{equation*}
  Z_{\hat{g}}(M) = \left\{ \theta \in \mathcal{P}(\mathcal{P}(\mathcal{Z} \times \mathcal{Z} \times \mathcal{Y} \times [0, 1])) \colon \int_{\mathcal{P}(\mathcal{Z} \times \mathcal{Z} \times \mathcal{Y} \times [0, 1])} \hat{g}(q) \, \theta(dq) \le M \right\}
\end{equation*}
is tight. Tightness of $\{ P^{\varepsilon, \Delta}, \varepsilon, \Delta > 0 \}$ follows from the bound

\begin{align}
\sup_{\epsilon\in(0,1]}\E [ \hat{g}(P^{\varepsilon, \Delta}) ]   &  =\sup_{\epsilon\in(0,1]}\E \left[
\int_{\mathcal{Z}\times\mathcal{Z}\times\mathcal{Y}\times[0,1]}\left(\left| z_{1}\right|^{2}+\left| z_{2}\right|^{2}+\left| y\right|^{2r}\right)
\mathrm{P}^{\epsilon,\Delta}(dz_{1} \,dz_{2} \,dy \,dt)\right] \nonumber\\
&  =\sup_{\epsilon\in(0,1]}\E\int_{0}^{1}\frac{1}{\Delta}%
\int_{t}^{t+\Delta}\left[\left| u^{\varepsilon}_{1}(s)\right| ^{2}+\left| u^{\varepsilon}_{2}(s)\right| ^{2}+\left| Y^{\varepsilon,u^{\varepsilon}}_{s}\right| ^{2r}\right]ds \,dt\nonumber\\
&  <\infty.\nonumber
\end{align}

Lastly the uniform integrability statement of Proposition \ref{P:tightness} follows from the last display and the following observation
\begin{align}
&\E\left[  \int_{(z_{1},z_{2},y)\in\mathcal{B}_{r,M}\times [0,1]}\left(\left| z_{1}\right|+\left| z_{2}\right|+\left| y\right|^{r}\right) \mathrm{P}%
^{\epsilon,\Delta}(dz_{1} \,dz_{2} \,dy \,dt)\right] \nonumber\\
&\qquad \leq\frac{C}{M}\E\left[
\int_{\mathcal{Z}\times\mathcal{Z}\times\mathcal{Y}\times[0,1]}\left(\left| z_{1}\right|^{2}+\left| z_{2}\right|^{2}+\left| y\right|^{2r}\right)\mathrm{P}^{\epsilon,\Delta}(dz_{1} \,dz_{2} \,dy \,dt)\right],\nonumber
\end{align}
for some unimportant constant $C<\infty$.

\subsubsection{Tightness of $\{ \eta^{\varepsilon, u^\varepsilon},\varepsilon>0 \}$ on $\mathcal{C}([0,1];\mathbb{R}^{n})$}\label{SS:tightnessProcess}

Next, we prove tightness of the family $\{ \eta^{\varepsilon, u^\varepsilon} \}$. It is sufficient to prove that for every $\zeta > 0$
\begin{equation} \label{E:tightness}
  \lim_{\rho \downarrow 0} \limsup_{\varepsilon \downarrow 0} \p \left[ \sup_{\substack{0 \le t_1 < t_2 \le 1 \\ \lvert t_2 - t_1 \rvert < \rho}} \lvert \eta_{t_2}^{\varepsilon, u^\varepsilon} - \eta_{t_1}^{\varepsilon, u^\varepsilon} \rvert > \zeta \right] = 0 .
\end{equation}
This proof is the main source of additional complexity as compared to the large deviations case. The proof depends on several technical lemmas which are stated and proved in the Appendix.

From \eqref{E:deviationsIntegral}, we have
\begin{align*}
  \eta_{t_2}^{\varepsilon, u^\varepsilon} - \eta_{t_1}^{\varepsilon, u^\varepsilon} &= \int_{t_1}^{t_2} \frac{\frac{\varepsilon}{\delta}b(X_s^{\varepsilon, u^\varepsilon}, Y_s^{\varepsilon, u^\varepsilon}) + c(X_s^{\varepsilon, u^\varepsilon}, Y_s^{\varepsilon, u^\varepsilon}) - \bar{\lambda}_1(\bar{X}_s)}{\sqrt{\varepsilon} h(\varepsilon)} \,ds \\
  &+ \int_{t_1}^{t_2} \sigma(X_s^{\varepsilon, u^\varepsilon}, Y_s^{\varepsilon, u^\varepsilon}) u_1^\varepsilon(s) \,ds
  + \frac{1}{h(\varepsilon)} \int_{t_1}^{t_2} \sigma(X_{s}^{\varepsilon, u^\varepsilon}, Y_{s}^{\varepsilon, u^\varepsilon}) \,dW_{s}. \notag
\end{align*}

We can rewrite the first term in the form
\begin{align} \label{E:lambdaSplit}
  &\int_{t_1}^{t_2} \frac{\frac{\varepsilon}{\delta}b(X_s^{\varepsilon, u^\varepsilon}, Y_s^{\varepsilon, u^\varepsilon}) + c(X_s^{\varepsilon, u^\varepsilon}, Y_s^{\varepsilon, u^\varepsilon}) - \bar{\lambda}_1(\bar{X}_s)}{\sqrt{\varepsilon} h(\varepsilon)} ds \\
  \notag &\ = \int_{t_1}^{t_2} \frac{\frac{\varepsilon}{\delta}b(X_s^{\varepsilon, u^\varepsilon}, Y_s^{\varepsilon, u^\varepsilon}) + c(X_s^{\varepsilon, u^\varepsilon}, Y_s^{\varepsilon, u^\varepsilon}) - \lambda_1(X_s^{\varepsilon, u^\varepsilon}, Y_s^{\varepsilon, u^\varepsilon})}{\sqrt{\varepsilon} h(\varepsilon)} ds \\
  \notag &\ + \int_{t_1}^{t_2} \frac{\lambda_1(X_s^{\varepsilon, u^\varepsilon}, Y_s^{\varepsilon, u^\varepsilon}) - \bar{\lambda}_1(X_s^{\varepsilon, u^\varepsilon})}{\sqrt{\varepsilon} h(\varepsilon)} ds + \int_{t_1}^{t_2} \frac{\bar{\lambda}_1(X_s^{\varepsilon, u^\varepsilon}) - \bar{\lambda}_1(\bar{X}_s)}{\sqrt{\varepsilon} h(\varepsilon)} ds.
\end{align}

Then we have
\begin{align*} \label{E:Pdecomp}
  &\p \left[ \sup_{\substack{0 \le t_1 < t_2 \le 1 \\ \lvert t_2 - t_1 \rvert < \rho}} \lvert \eta_{t_2}^{\varepsilon, u^\varepsilon} - \eta_{t_1}^{\varepsilon, u^\varepsilon} \rvert > \zeta \right]  \\
  &\le \p \left[ \sup_{\substack{0 \le t_1 < t_2 \le 1 \\ \lvert t_2 - t_1 \rvert < \rho}} \left\lvert \int_{t_1}^{t_2} \frac{\frac{\varepsilon}{\delta}b(X_s^{\varepsilon, u^\varepsilon}, Y_s^{\varepsilon, u^\varepsilon}) + c(X_s^{\varepsilon, u^\varepsilon}, Y_s^{\varepsilon, u^\varepsilon}) - \lambda_1(X_s^{\varepsilon, u^\varepsilon}, Y_s^{\varepsilon, u^\varepsilon})}{\sqrt{\varepsilon} h(\varepsilon)} ds \right\rvert > \frac{\zeta}{5} \right] \notag \\
  &+ \p \left[ \sup_{\substack{0 \le t_1 < t_2 \le 1 \\ \lvert t_2 - t_1 \rvert < \rho}} \left\lvert \int_{t_1}^{t_2} \frac{\lambda_1(X_s^{\varepsilon, u^\varepsilon}, Y_s^{\varepsilon, u^\varepsilon}) - \bar{\lambda}_1(X_s^{\varepsilon, u^\varepsilon})}{\sqrt{\varepsilon} h(\varepsilon)} ds \right\rvert > \frac{\zeta}{5} \right] \notag \\
  &+ \p \left[ \sup_{\substack{0 \le t_1 < t_2 \le 1 \\ \lvert t_2 - t_1 \rvert < \rho}} \left\lvert \int_{t_1}^{t_2} \frac{\bar{\lambda}_1(X_s^{\varepsilon, u^\varepsilon}) - \bar{\lambda}_1(\bar{X}_s)}{\sqrt{\varepsilon} h(\varepsilon)} ds \right\rvert > \frac{\zeta}{5} \right] \notag \\
  &+ \p \left[ \sup_{\substack{0 \le t_1 < t_2 \le 1 \\ \lvert t_2 - t_1 \rvert < \rho}} \left\lvert \int_{t_1}^{t_2} \sigma(X_s^{\varepsilon, u^\varepsilon}, Y_s^{\varepsilon, u^\varepsilon}) u_1^\varepsilon(s) \,ds \right\rvert > \frac{\zeta}{5} \right]  \notag \\
  &+ \p \left[ \sup_{\substack{0 \le t_1 < t_2 \le 1 \\ \lvert t_2 - t_1 \rvert < \rho}} \left\lvert \frac{1}{h(\varepsilon)} \int_{t_1}^{t_2} \sigma(X_{s}^{\varepsilon, u^\varepsilon}, Y_{s}^{\varepsilon, u^\varepsilon}) \,dW_{s} \right\rvert > \frac{\zeta}{5} \right] \notag \\
  &= \sum_{i=1}^5 J_i^{\varepsilon, \rho}. \notag
\end{align*}

By Lemma \ref{L:coefs}, \ref{L:lambda}, \ref{L:lambdabar}, \ref{L:productBound} we have for $i=1$, $2$, $3$, $4$ respectively that
\begin{equation*}
  \lim_{\rho \downarrow 0} \limsup_{\varepsilon \downarrow 0} J_i^{\varepsilon, \rho} = 0 .
\end{equation*}
It remains to study the term $J_5^{\varepsilon, \rho}$. By the conditions on $\sigma$ and Lemma \ref{L:Ygrowth},
\begin{equation*}
  M_t = \int_0^t \sigma(X_{s}^{\varepsilon, u^\varepsilon}, Y_{s}^{\varepsilon, u^\varepsilon}) \,dW_{s}
\end{equation*}
is a local square integrable martingale with continuous paths. Then, using again Lemma \ref{L:Ygrowth}, we have for a constant $C<\infty$ that may change from line to line and  for $\nu>0$ small enough such that $q_{\sigma}(1+\nu)<r$, we have
\begin{align*}
  &\p \left[ \sup_{\substack{0 \le t_1 \le t \le t_{1}+\rho }} \left\lvert \int_{t_1}^{t} \sigma(X_{s}^{\varepsilon, u^\varepsilon}, Y_{s}^{\varepsilon, u^\varepsilon}) \,dW_{s} \right\rvert > h(\varepsilon) \frac{\zeta}{5} \right] \nonumber\\
  &\quad\leq C \left(h(\varepsilon)\zeta\right)^{-2(1+\nu)} \E \left[ \sup_{\substack{0\le t_1 \le t \le t_{1}+\rho }} \left\lvert \int_{t_1}^{t} \sigma(X_{s}^{\varepsilon, u^\varepsilon}, Y_{s}^{\varepsilon, u^\varepsilon}) \,dW_{s} \right\rvert^{2(1+\nu)}  \right]\nonumber\\
  &\quad\leq C \left(h(\varepsilon)\zeta\right)^{-2(1+\nu)}  \E \left[ \sup_{\substack{0 \le t_1 \le t \le t_{1}+\rho }} \left\lvert \int_{t_1}^{t} \left|\sigma(X_{s}^{\varepsilon, u^\varepsilon}, Y_{s}^{\varepsilon, u^\varepsilon})\right|^{2} \,ds \right\rvert^{(1+\nu)}  \right]%\nonumber\\
\end{align*}
from which the result follows by Lemma \ref{L:productBound}. %In the last step of the previous display we used H\"{o}lder's inequality.
With this, the proof of \eqref{E:tightness} is completed.

\subsection{Proof of existence of viable pair}\label{SS:ExistenceViablePair}

In the previous section, we have shown that the family of processes $\{ (\eta^{\varepsilon, u^\varepsilon}, P^{\varepsilon, \Delta}),\ \varepsilon > 0 \}$ is tight. It follows that for any subsequence of $\varepsilon$ converging to 0, there exists a subsubsequence of $(\eta^{\varepsilon, u^\varepsilon}, P^{\varepsilon, \Delta})$ which is convergent in distribution to some limit $(\bar{\eta}, \bar{P})$. The goal of this section is to show that $(\bar{\eta}, \bar{P})$ is a viable pair with respect to $(\theta_{1},\opL_{1,x})$ according to Definition \ref{D:ViablePair}. For this purpose we use the martingale problem formulation.

By the Skorokhod Representation Theorem, we may assume that there exists a probability space in which the desired convergence occurs w.p.1. By the proof of tightness for $\{P^{\varepsilon, \Delta} \}$ and Fatou's lemma,
\begin{equation*}
\E \int_{\mathcal{Z} \times \mathcal{Z} \times \mathcal{Y} \times [0, 1]} \left[ \lvert z_1\rvert^2 + \lvert z_2\rvert^2 + \lvert y\rvert^{2r}\right] \bar{P}(dz_1\,dz_2\,dy\,dt) < \infty
\end{equation*}
which then implies that $\int_{\mathcal{Z} \times \mathcal{Z} \times \mathcal{Y} \times [0, 1]} \left[ \lvert z_1\rvert^2 + \lvert z_2\rvert^2 + \lvert y\rvert^{2r}\right] \bar{P}(dz_1\,dz_2\,dy\,dt) < \infty$ w.p.1. Here $r$ is the order of polynomial bound in $|y|$ of the $\theta_{1}$ function.

Therefore, to show that the limit point $(\bar{\eta}, \bar{P})$ is a viable pair, we must show that it satisfies equations \eqref{E:viableSDE}, \eqref{E:viableCenter}, and \eqref{E:viableLebesgue}.

We begin with \eqref{E:viableSDE}. Let $p_1$ and $p_2$ be positive integers. Let $F$ be a real valued, smooth function with compact support on $\mathbb{R}^n$.  Let $\phi_j$, $j = 1, \dots, p_1$, be real valued, smooth functions with compact support on $\mathcal{Z} \times \mathcal{Z} \times \mathcal{Y} \times [0, 1]$. Let $S$, $T$, and $t_i$, $i = 1, \dots, p_2$, be nonnegative real numbers such that $t_i \le S <  S + T \le 1$. Let $\zeta$ be a real valued, bounded and continuous function with compact support on $(\mathbb{R}^n)^{p_2} \times \mathbb{R}^{p_1 p_2}$. For a measure $\hat{r} \in \mathcal{P}(\mathcal{Z} \times \mathcal{Z} \times \mathcal{Y} \times [0,1])$ and $t \in [0,1]$, define
\begin{equation*}
  (\hat{r}, \phi_j)_t = \int_{\mathcal{Z} \times \mathcal{Z} \times \mathcal{Y} \times [0,t]} \phi_j(z_1, z_2, y, s) \,\hat{r}(dz_1 \,dz_2 \,dy \,ds).
\end{equation*}

Define the operator $\bar{\opL}_t^{\varepsilon, \Delta}$ by
\begin{equation*}
\bar{\opL}_t^{\varepsilon, \Delta}F(\eta) = \int_{\mathcal{Z} \times \mathcal{Z} \times \mathcal{Y}}  \big( \nabla F(\eta) \big) \theta_1(\bar{X}_t, \eta, y, z_1, z_2) \,P_t^{\varepsilon, \Delta}(dz_1 \,dz_2 \,dy)
\end{equation*}
where
\begin{equation*}
  P_t^{\varepsilon, \Delta}(dz_1 \,dz_2 \,dy) = \frac{1}{\Delta} \int_t^{t+\Delta} 1_{dz_1}(u_1^\varepsilon(s)) 1_{dz_2}(u_2^\varepsilon(s)) 1_{dy}(Y_s^{\varepsilon, u^\varepsilon}) \,ds.
\end{equation*}

Then to prove \eqref{E:viableSDE}, it is sufficient to prove that as $\varepsilon \downarrow 0$,
\begin{equation}\label{E:martingaleProblem}
  \E \left[ \zeta(\eta_{t_i}^{\varepsilon, u^\varepsilon}, (P^{\varepsilon, \Delta}, \phi_j)_{t_i}, i \le p_2, j \le p_1) \left[ F(\eta_{S + T}^{\varepsilon, u^\varepsilon}) - F({\eta_{S}^{\varepsilon, u^\varepsilon}}) - \int_S^{S + T} \bar{\opL}_t^{\varepsilon, \Delta}F(\eta_{t}^{\varepsilon, u^\varepsilon}) \,dt \right] \right] \to 0
\end{equation}
and
\begin{equation}\label{E:PConvergence}
  \int_S^{S + T} \bar{\opL}_t^{\varepsilon, \Delta}F(\eta_{t}^{\varepsilon, u^\varepsilon}) \,dt - \int_{\mathcal{Z} \times \mathcal{Z} \times \mathcal{Y} \times [S, S + T]} \big( \nabla F( \bar{\eta}_t) \big) \theta_1(\bar{X}_t, \bar{\eta}_t, y, z_1, z_2) \, \bar{P}(dz_1 \,dz_2 \,dy \,dt) \to 0.
\end{equation}

For every real valued, continuous function $\phi$ with compact support and $t \in [0, 1]$,
\begin{equation*}
  (P^{\varepsilon, \Delta}, \phi)_t \to (\bar{P}, \phi)_t \qquad \text{w.p.1.}
\end{equation*}

\begin{lemma} \label{L:viableConvergence}
Let $S > 0$ and $T > 0$ be positive numbers such that $S + T  \le 1$. Consider a continuous function $\xi \colon \mathbb{R}^n \times \mathbb{R}^n \times \mathcal{Y} \times \mathcal{Z} \times \mathcal{Z} \to \mathbb{R}$ that is bounded in the first argument, affine in the second argument, not growing faster than $|y|^{r}$ in the third argument and affine in the last two arguments. Assume that $(\eta^{\varepsilon, u^\varepsilon}, P^{\varepsilon, \Delta}) \to (\bar{\eta}, \bar{P})$ in distribution for some subsequence of $\varepsilon \downarrow 0$, and that Conditions \ref{C:growth} and \ref{C:ergodic} (and in Regime 1, Condition \ref{C:center}) hold. Then the following limits are valid in distribution along this subsequence:
\begin{multline*}
  \int_{\mathcal{Z} \times \mathcal{Z} \times \mathcal{Y} \times [S, S + T]} \xi(\bar{X}_t, \eta_t^{\varepsilon, u^\varepsilon}, y, z_1, z_2) \, P^{\varepsilon, \Delta} (dz_1 \,dz_2 \,dy \,dt) \\
  \to \int_{\mathcal{Z} \times \mathcal{Z} \times \mathcal{Y} \times [S, S+T]} \xi(\bar{X}_t, \bar{\eta}_t, y, z_1, z_2) \, \bar{P}(dz_1 \,dz_2 \,dy \,dt)
\end{multline*}
and
\begin{multline*}
  \int_S^{S+T} \xi(X_t^{\varepsilon, u^\varepsilon}, \eta_t^{\varepsilon, u^\varepsilon}, Y_t^{\varepsilon, u^\varepsilon}, u_1^\varepsilon(t), u_2^\varepsilon(t) ) \,dt \\
  {} - \int_{\mathcal{Z} \times \mathcal{Z} \times \mathcal{Y} \times [S, S + T]} \xi(\bar{X}_t, \eta_t^{\varepsilon, u^\varepsilon}, y, z_1, z_2) \,P^{\varepsilon, \Delta} (dz_1 \,dz_2 \,dy \,dt) \to 0.
\end{multline*}
\end{lemma}

Lemma \ref{L:viableConvergence} is similar to Lemma 3.2 from \cite{DS12} with the difference however that the function $\xi$ is not bounded in $y$.  The proof of Lemma \ref{L:viableConvergence} follows the same lines as that of Lemma 3.2 from \cite{DS12}, where here we need to make use of the uniform integrability of  $P^{\varepsilon, \Delta} $ with respect to both $(z_{1},z_{2})$ and $y$ from the second part of Proposition \ref{P:tightness}, in the same way that the uniform integrability with respect to just the control $z$ was used in \cite{DS12}. The details are omitted.

We apply this lemma with $\xi(x, \eta, y, z_1, z_2) = \big( \nabla F(\eta) \big) \theta_1(x, \eta, y, z_1, z_2)$. The first statement of Lemma \ref{L:viableConvergence} is equivalent to \eqref{E:PConvergence}, and the second is equivalent (after applying the It\^{o} formula to $F(\eta)$) to \eqref{E:martingaleProblem}, which proves
 \eqref{E:viableSDE}.

To prove \eqref{E:viableCenter}, introduce the operator $\tilde{\opL}_{z_1, z_2, x}^\varepsilon$ for functions $F \in \mathcal{C}^2(\mathcal{Y})$ defined by
\begin{align*}
  \tilde{\opL}_{z_1, z_2, x}^\varepsilon F(y) &= \frac{1}{\delta} \big( \nabla F(y) \big) \left[ \frac{\varepsilon}{\delta} f(x, y) + g(x, y) + \sqrt{\varepsilon} h(\varepsilon) \tau_1(x, y) z_1 + \sqrt{\varepsilon} h(\varepsilon) \tau_2(x, y) z_2 \right] \notag \\
  &+ \frac{\varepsilon}{\delta^2} \frac{1}{2} (\tau_1 \tau_1^\T + \tau_2 \tau_2^\T)(x,y) : \nabla \nabla F(y).
\end{align*}

Consider $\{F_\ell : \mathcal{Y} \to \mathbb{R}, \ell \in \mathbb{N} \}$ to be a smooth and dense family of bounded functions with bounded derivatives in $\mathcal{C}^2(\mathcal{Y})$. Then it is easy to see that
\begin{equation*}
  M_t^{\varepsilon} = F_\ell(Y_t^{\varepsilon, u^\varepsilon}) - F_\ell(y_0) - \int_0^t \tilde{\opL}_{u_1^\varepsilon(s), u_2^\varepsilon(s), X_s^{\varepsilon, u^\varepsilon}}^\varepsilon F_\ell(Y_s^{\varepsilon, u^\varepsilon}) \,ds
\end{equation*}
is an $\mathscr{F}_t$ martingale. Let $G(\varepsilon) = \delta^2 / \varepsilon$ and notice that $G(\varepsilon) \tilde{\opL}_{z_1, z_2, x}^\varepsilon$ converges to $\opL_{1, x}$ as $\varepsilon \downarrow 0$. Next, we define the operator
\begin{equation*}
  \mathcal{G}_{z_1, z_2, x} F_{\ell}(y) =  \big( \nabla F_{\ell}(y) \big) ( \tau_1(x, y) z_1 + \tau_2(x, y) z_2)
\end{equation*}
and write
\begin{align} \label{E:Gequation}
  &G(\varepsilon) M_t^{\varepsilon} - G(\varepsilon) (F_\ell(Y_t^{\varepsilon, u^\varepsilon}) - F_\ell(y_0)) \\
  &\quad - G(\varepsilon) \left[ \int_0^t \frac{1}{\Delta} \left[ \int_s^{s+\Delta} \tilde{\opL}_{u_1^\varepsilon(\rho), u_2^\varepsilon(\rho), X_\rho^{\varepsilon, u^\varepsilon}}^\varepsilon F_\ell(Y_\rho^{\varepsilon, u^\varepsilon}) \,d\rho \right] \,ds \right. \notag \\
  & \quad \left. - \int_0^t \tilde{\opL}_{u_1^\varepsilon(s), u_2^\varepsilon(s), X_s^{\varepsilon, u^\varepsilon}}^\varepsilon F_\ell(Y_s^{\varepsilon, u^\varepsilon}) \,ds \right] \notag \\
  & \notag = - \frac{\delta}{\varepsilon} \sqrt{\varepsilon} h(\varepsilon) \int_0^t \frac{1}{\Delta} \left[ \int_s^{s+\Delta} \left[ \mathcal{G}_{u_1^\varepsilon(\rho), u_2^\varepsilon(\rho), X_\rho^{\varepsilon, u^\varepsilon}} F_\ell(Y_\rho^{\varepsilon, u^\varepsilon}) - \mathcal{G}_{u_1^\varepsilon(\rho), u_2^\varepsilon(\rho), X_s^{\varepsilon, u^\varepsilon}} F_\ell(Y_\rho^{\varepsilon, u^\varepsilon}) \right] \,d\rho \right] \,ds \\
  & \notag \quad - \frac{\delta}{\varepsilon} \sqrt{\varepsilon} h(\varepsilon) \int_{\mathcal{Z} \times \mathcal{Z} \times \mathcal{Y} \times [0, t]} \mathcal{G}_{z_1, z_2, X_s^{\varepsilon, u^\varepsilon}} F_\ell(y) \,P^{\varepsilon, \Delta}(dz_1\,dz_2\,dy\,ds) \\
  & \notag \quad - \frac{\delta}{\varepsilon} \int_o^t \frac{1}{\Delta} \left[ \int_s^{s+\Delta} \big( \nabla F_\ell(Y_\rho^{\varepsilon, u^\varepsilon}) \big) \left[ g(X_\rho^{\varepsilon, u^\varepsilon}, Y_\rho^{\varepsilon, u^\varepsilon}) - g(X_s^{\varepsilon, u^\varepsilon}, Y_\rho^{\varepsilon, u^\varepsilon}) \right] \,d\rho \right] \,ds \\
  & \notag \quad -\frac{\delta}{\varepsilon} \int_{\mathcal{Z} \times \mathcal{Z} \times \mathcal{Y} \times [0,t]} \big( \nabla F_\ell(y) \big) g(X_s^{\varepsilon, u^\varepsilon}, y) \,P^{\varepsilon, \Delta}(dz_1\, dz_2 \,dy \,ds) \\
  & \notag \quad - \int_0^t \frac{1}{\Delta} \left[ \int_s^{s+\Delta} \left[\opL_{1, X_\rho^{\varepsilon, u^\varepsilon}} F_\ell(Y_\rho^{\varepsilon, u^\varepsilon}) - \opL_{1, X_s^{\varepsilon, u^\varepsilon}} F_\ell(Y_\rho^{\varepsilon, u^\varepsilon}) \right] \,d\rho \right] \,ds  \\
  & \notag \quad - \int_{\mathcal{Z} \times \mathcal{Z} \times \mathcal{Y} \times [0, t]}  \opL_{1, X_s^{\varepsilon, u^\varepsilon}} F_\ell(y) \,P^{\varepsilon, \Delta}(dz_1\,dz_2\,dy\,ds).
\end{align}

Now consider each of these terms as $\varepsilon \downarrow 0$. The left hand side of \eqref{E:Gequation} goes to zero since:
\begin{enumerate}[(a)]
\item $M_t^\varepsilon$ is square integrable, so $G(\varepsilon) M_t^\varepsilon \downarrow 0$ in probability as $\varepsilon \downarrow 0$,
\item $F_\ell$ is bounded, $G(\varepsilon) \left[ F_\ell(Y_t^{\varepsilon, u^\varepsilon}) - F_\ell(y_0) \right]$ converges to zero uniformly as $\varepsilon \downarrow 0$, and
\item $\Delta \downarrow 0$ as $\varepsilon \downarrow 0$,
\begin{equation*}
G(\varepsilon) \left[ \int_0^t \frac{1}{\Delta} \left[ \int_s^{s+\Delta} \tilde{\opL}_{u_1^\varepsilon(\rho), u_2^\varepsilon(\rho), X_\rho^{\varepsilon, u^\varepsilon}}^\varepsilon F_\ell(Y_\rho^{\varepsilon, u^\varepsilon}) \,d\rho \right] \,ds - \int_0^t \tilde{\opL}_{u_1^\varepsilon(s), u_2^\varepsilon(s), X_s^{\varepsilon, u^\varepsilon}}^\varepsilon F_\ell(Y_s^{\varepsilon, u^\varepsilon}) \,ds \right]
\end{equation*}
converges to zero in probability.
\end{enumerate}

We next study the right hand side of \eqref{E:Gequation}. Tightness of $\{ X^{\varepsilon, u^\varepsilon},\varepsilon>0 \}$ implies that the first term, third term, and fifth term on the right side converge to zero in probability as $\varepsilon \downarrow 0$. (Tightness of $\{ X^{\varepsilon, u^\varepsilon},\varepsilon>0 \}$ follows immediately from tightness of $\{ \eta^{\varepsilon, u^\varepsilon},\varepsilon>0 \}$ by \eqref{E:controlledDeviations}.)

Uniform integrability of $P^{\varepsilon, \Delta}$ and the fact that $\delta / \varepsilon \downarrow 0$ imply that the second and fourth terms on the right side converge to zero in probability as $\varepsilon \downarrow 0$.

Therefore,
\begin{equation*}
  \int_{\mathcal{Z} \times \mathcal{Z} \times \mathcal{Y} \times [0, t]} \opL_{1, X_s^{\varepsilon, u^\varepsilon}} F_\ell(y) P^{\varepsilon, \Delta}(dz_1\,dz_2\,dy\,ds) \to 0 \text{ in probability as } \varepsilon \downarrow 0.
 \end{equation*}

This implies \eqref{E:viableCenter} by continuity in $t$ and density of $\{F_\ell, \ell\in \mathbb{N}$\}.

Proof of \eqref{E:viableLebesgue} is identical to \cite{DS12} or \cite{S13}. More explicitly, by the fact that
$P^{\varepsilon, \Delta}(\mathcal{Z} \times \mathcal{Z} \times \mathcal{Y} \times [0, t]) = t$, along with $P(\mathcal{Z} \times \mathcal{Z} \times \mathcal{Y} \times\{t\}) = 0$ and the continuity of the mapping $t \to P(\mathcal{Z} \times \mathcal{Z} \times \mathcal{Y} \times [0, t])$, the property holds.

\subsection{Proof of Laplace principle lower bound}
We now prove the Laplace principle lower bound. %Let $\mathcal{V}$ be the space of viable pairs $(\xi, P)$ with respect to $(\theta_1, \opL_{1,x})$ as defined by Definition \ref{D:ViablePair}.
We want to show that for all bounded, continuous functions $a$ mapping $\mathcal{C}([0,1]; \mathbb{R}^n)$ into $\mathbb{R}$,
\begin{align*}
  &\liminf_{\varepsilon \downarrow 0} - \frac{1}{h^2(\varepsilon)} \log \E \left[ \exp \left\{ - h^2(\varepsilon) a(\eta^\varepsilon) \right\} \right] \nonumber\\
  &\qquad\ge \inf_{(\xi, P) \in \mathcal{V}(\theta_1, \opL_{1,x})} \left[ \frac{1}{2} \int \left[ \lvert z_1\rvert^2 + \lvert z_2\rvert^2 \right] P(dz_1\, dz_2\, dy \,ds) + a(\xi) \right].
\end{align*}
It is sufficient to prove the lower limit along any subsequence such that
\begin{equation*}
  - \frac{1}{h^2(\varepsilon)} \log \E \left[ \exp \left\{ - h^2(\varepsilon) a(\eta^\varepsilon) \right\} \right]
\end{equation*}
converges. Such a subsequence exists because $\lvert - 1/h^2(\varepsilon) \log \E \left[ \exp \left\{ - h^2(\varepsilon) a(\eta^\varepsilon) \right\} \right] \rvert \le \lVert a\rVert_\infty$.
By Lemma \ref{L:uBound}, we may assume that
\begin{equation*}
  \sup_{\varepsilon > 0} \E\int_0^1 \lvert u^\varepsilon(s) \rvert^2 \,ds \le N.
\end{equation*}
for some constant $N$.

We construct the family of occupation measures $P^{\varepsilon, \Delta}$, and the family $\{(\eta^{\varepsilon, u^\varepsilon}, P^{\varepsilon, \Delta}), \varepsilon >0 \}$ is tight. Hence, for any subsequence of $\varepsilon \downarrow 0$ there is a further subsequence for which
\begin{equation*}
  (\eta^{\varepsilon, u^\varepsilon}, P^{\varepsilon, \Delta}) \to (\bar{\eta}, \bar{P}) \text{ in distribution}
\end{equation*}
with $(\bar{\eta}, \bar{P}) \in \mathcal{V}(\theta_1, \opL_{1,x})$. By Fatou's lemma, we then obtain
\begin{align*}
  &\liminf_{\varepsilon \downarrow 0} \left( - \frac{1}{h^2(\varepsilon)} \log \E \left[ \exp \left\{ - h^2(\varepsilon) a(\eta^\varepsilon) \right\} \right] \right) \\
    &\quad \ge \liminf_{\varepsilon \downarrow 0} \left( \E \left[ \frac{1}{2} \int_0^1 \lvert u^\varepsilon(s)\rvert^2 \,ds + a(\eta^{\varepsilon,u^\varepsilon}) \right] - \varepsilon \right) \notag \\
    &\quad \ge \liminf_{\varepsilon \downarrow 0} \left( \E \left[ \frac{1}{2} \int_0^1 \frac{1}{\Delta} \int_t^{t+\Delta} \lvert u^\varepsilon(s)\rvert^2 \,ds \,dt + a(\eta^{\varepsilon,u^\varepsilon}) \right] \right) \notag \\
    &\quad = \liminf_{\varepsilon \downarrow 0} \left( \E \left[ \frac{1}{2} \int_{\mathcal{Z} \times \mathcal{Z} \times \mathcal{Y} \times [0,1]} \left[ \lvert z_1\rvert^2 +  \lvert z_2\rvert^2 \right] \,P^{\varepsilon, \Delta} (dz_1\,dz_2\,dy\,dt) + a(\eta^{\varepsilon,u^\varepsilon}) \right] \right) \notag \\
    &\quad \ge \E \left[ \frac{1}{2} \int_{\mathcal{Z} \times \mathcal{Z} \times \mathcal{Y} \times [0,1]} \left[ \lvert z_1\rvert^2 +  \lvert z_2\rvert^2 \right] \,\bar{P} (dz_1\,dz_2\,dy\,dt) + a(\bar{\eta}) \right] \notag \\
    &\quad \ge \inf_{(\xi, P) \in \mathcal{V}(\theta_1, \opL_{1,x})} \left\{ \frac{1}{2} \int_{\mathcal{Z} \times \mathcal{Z} \times \mathcal{Y} \times [0,1]} \left[ \lvert z_1\rvert^2 +  \lvert z_2\rvert^2 \right] \,P(dz_1\,dz_2\,dy\,dt) + a(\xi) \right\}. \notag
\end{align*}

This concludes the proof of the Laplace principle lower bound.
\subsection{Proof of compactness of level sets of $S(\cdot)$}

We want to prove that for each $s < \infty$, the set
\begin{equation*}
  \Xi_s = \{ \xi \in \mathcal{C}([0,1]; \mathbb{R}^n) : S(\xi) \le s \}
\end{equation*}
is a compact subset of $\mathcal{C}([0,1]; \mathbb{R}^n)$. The proof is analogous to the proof of the lower bound. We need to show precompactness of $\Xi_s$ and that it is a closed set.

Precompactness of the pair $\{ (\xi^{n},P^n), n>0\}$ follows by standard arguments, see for example \cite{DS12}. Next we must show that the limit of a sequence of viable pairs is a viable pair. Fix $K < \infty$ and consider any convergent sequence $\{ (\xi^n, P^n), n > 0\}$ such that for every $n > 0$, $(\xi^n, P^n) \in \mathcal{V}(\theta_1, \opL_{1,x})$ and
\begin{equation*}
  \int_{\mathcal{Z} \times \mathcal{Z} \times \mathcal{Y} \times [0,1]} \left[ \lvert z_1\rvert^2 +  \lvert z_2\rvert^2 +  \lvert y\rvert^{2r}\right] \,P^n (dz_1\,dz_2\,dy\,dt) < K,
\end{equation*}
where $r$ is the order of the polynomial bound in $\lvert y \rvert$ of $\theta_1$.

Since $(\xi^n, P^n)$ is a viable pair, we get that
\begin{equation*}
  \xi^n_t = \int_{\mathcal{Z} \times \mathcal{Z} \times \mathcal{Y} \times [0,t]} \theta_1(\bar{X}_s, \xi^n_s, y, z_1, z_2) \,P^n(dz_1 \,dz_2 \,dy \,ds)
\end{equation*}
and
\begin{equation*}
  \int_0^t \int_{\mathcal{Z} \times \mathcal{Z} \times \mathcal{Y}} \opL_{1, \bar{X}_s} F(y) \,P^n(dz_1 \,dz_2 \,dy \,ds) = 0
\end{equation*}
for every $t \in [0,1]$ and  every $F \in \mathcal{C}^2(\mathcal{Y})$. Then by the convergence of $(\xi^n, P^n)$ to $(\xi, P)$, we get that $(\xi, P) \in \mathcal{V}(\theta_1, \opL_{1,x})$.

Finally, we must prove lower semicontinuity, that is
\begin{equation*}
\liminf_{n \to \infty} S(\xi^n) \ge S(\xi)
\end{equation*}

Without loss of generality, we may assume that there is some $M<\infty$ such that $\liminf_{n \to \infty} S(\xi^n)\leq M$. Also, by the definition of $S(\xi^n)$, we obtain that one can find measures $\{P^{n},n<\infty\}$ such that $(\xi^{n}, P^{n}) \in \mathcal{V}(\theta_1, \opL_{1,x})$,
\[
\sup_{n<\infty}\int_{\mathcal{Z} \times \mathcal{Z} \times \mathcal{Y} \times [0, 1]} \left[ \lvert z_1 \rvert^2 + \lvert z_2 \rvert^2 \right] \,P^n(dz_1\,dz_2\,dy\,dt)<M+1
\]
and
\[
S(\xi^n)\geq \frac{1}{2} \int_{\mathcal{Z} \times \mathcal{Z} \times \mathcal{Y} \times [0, 1]} \left[ \lvert z_1 \rvert^2 + \lvert z_2 \rvert^2 \right] \,P^n(dz_1\,dz_2\,dy\,dt)-\frac{1}{n}.
\]

Then by Fatou's lemma we have
\begin{align*}
   &\liminf_{n \to \infty} S(\xi^n) \geq \liminf_{n \to \infty} \left[\frac{1}{2} \int_{\mathcal{Z} \times \mathcal{Z} \times \mathcal{Y} \times [0, 1]} \left[ \lvert z_1 \rvert^2 + \lvert z_2 \rvert^2 \right] \,P^n(dz_1\,dz_2\,dy\,dt)-\frac{1}{n}\right] \\
   \notag &\quad\ge \frac{1}{2} \int_{\mathcal{Z} \times \mathcal{Z} \times \mathcal{Y} \times [0, 1]} \left[ \lvert z_1 \rvert^2 + \lvert z_2 \rvert^2 \right] \,P(dz_1\,dz_2\,dy\,dt) \\
    \notag &\quad\ge \inf_{(\xi, P) \in \mathcal{V}(\theta_1, \opL_{1,x})} \left\{\frac{1}{2} \int_{\mathcal{Z} \times \mathcal{Z} \times \mathcal{Y} \times [0, 1]} \left[ \lvert z_1 \rvert^2 + \lvert z_2 \rvert^2 \right] \,P(dz_1\,dz_2\,dy\,dt)\right\}
   = S(\xi) .
\end{align*}

\subsection{Proof of Laplace principle upper bound and representation formula}

The first step is to establish the equivalence of the control formulation to the relaxed control formulation, as in \cite{DS12}. Let us briefly recall how this is done.

The action functional $S(\xi)$ can be written in terms of a local action functional, i.e.,
\begin{equation*}
  S(\xi) = \int_0^1 L^{rel}(\bar{X}_s, \xi_s, \dot{\xi}_s) \,ds.
\end{equation*}

This follows from the definition of a viable pair by setting
\begin{equation*}\label{E:relaxLocalRate}
  L^{rel}(x, \eta, \beta) = \inf_{P \in \mathcal{A}_{x, \eta, \beta}^{rel}} \int_{\mathcal{Z} \times \mathcal{Z} \times \mathcal{Y}} \frac{1}{2} \left[ \lvert z_1\rvert^2 + \lvert z_2\rvert^2 \right] \,P(dz_1\, dz_2\, dy)
\end{equation*}
where
\begin{align*}
  & \mathcal{A}_{x, \eta, \beta}^{rel} = \left\{ P \in \mathcal{P}(\mathcal{Z} \times \mathcal{Z} \times \mathcal{Y}) : \int_{\mathcal{Z} \times \mathcal{Z} \times \mathcal{Y}} \opL_{1, x} F(y) \,P(dz_1\,dz_2\,dy) = 0 \right. \\
  &\left.\qquad \text{ for all } F \in \mathcal{C}^2(\mathcal{Y}),
   \int_{\mathcal{Z} \times \mathcal{Z} \times \mathcal{Y}} \left[ \vert z_1\rvert^2 + \lvert z_2\rvert^2 + \lvert y\rvert^{2r}\right] \,P(dz_1\,dz_2\,dy) < \infty, \right.\\
  &\left.\qquad \text{ and } \beta = \int_{\mathcal{Z} \times \mathcal{Z} \times \mathcal{Y}} \theta_1(x, \eta, y, z_1, z_2) \,P(dz_1\,dz_2\,dy) \right\}. \notag
\end{align*}
The constant $r$ in the above expression is the order of the polynomial bound of $\lvert y \rvert$ in $\theta_1$.

Note that any measure $P \in \mathcal{P}(\mathcal{Z} \times \mathcal{Z} \times \mathcal{Y})$ can be decomposed in the form
\begin{equation}\label{E:measureDecomposition}
  P(dz_1\,dz_2\,dy) = \nu(dz_1\,dz_2|y) \mu(dy)
\end{equation}
where $\mu$ is a probability measure on $\mathcal{Y}$ and $\nu$ is a stochastic kernel on $\mathcal{Z} \times \mathcal{Z}$ given $\mathcal{Y}$. Following the terminology of \cite{DS12}, we refer to this as a ``relaxed'' formulation.

Inserting \eqref{E:measureDecomposition} into \eqref{E:viableCenter} and noticing that $\opL_{1,x}$ does not depend on the control variables, we obtain that for every $F \in \mathcal{C}^2(\mathcal{Y})$,
\[
  \int_\mathcal{Y} \opL_{1, x} F(y) \,\mu(dy) = 0.
\]

The nondegeneracy of $(\tau_1 \tau_1^\T + \tau_2\tau_2^\T)$ and the previous equation show that $\mu(dy)$ is the unique invariant measure corresponding to the operator $\opL_{1, x}$ (i.e.\ $\mu(dy) = \mu_{1,x}(dy)$).

Because the cost is convex in $z$ and $\theta_1$ is affine in $z$, the relaxed control formulation is equivalent to the ordinary control formulation of the local rate function:
\begin{equation}\label{E:ordLocalRate}
  L^o(x, \eta, \beta) = \inf_{(v, \mu) \in \mathcal{A}_{x,\eta,\beta}^o} \frac{1}{2} \int_\mathcal{Y} \lvert v(y) \rvert^2 \,\mu_{1, x}(dy)
\end{equation}
where
\begin{align*}
  \mathcal{A}_{x, \eta,\beta}^o &= \left\{ v(\cdot) = (v_1(\cdot), v_2(\cdot)) \colon \mathcal{Y} \to \mathbb{R}^{2m}, \mu \in \mathcal{P}(\mathcal{Y}) \colon (v, \mu) \text{ satisfy } \right. \\
  & \left. \int_\mathcal{Y} \opL_{1, x} F(y) \,\mu_{1,x}(dy) = 0 \text{ for all } F \in \mathcal{C}^2(\mathcal{Y}), \int_\mathcal{Y} \left[\lvert v(y)\rvert^2 +|y|^{2r}\right]\,\mu(dy) < \infty \right. \notag \\
  &\left. \text{ and } \beta = \int_\mathcal{Y} \theta_1(x, \eta, y, v_1(y), v_2(y)) \,\mu_{1,x}(dy) \right\}. \notag
\end{align*}

The equivalence of $L^{rel}(x, \eta, \beta)$ and $L^o(x, \eta, \beta)$ follows from Jensen's inequality and the fact that $\theta_1(x, \eta, y, z_1, z_2)$ and $\opL_{1, x}$ are affine in $z_1$ and $z_2$.

The following result is a key statement for the equivalence of Theorems \ref{T:main} and \ref{T:control}.

\begin{theorem} \label{T:controlRepresentation}
Under Conditions \ref{C:growth}, \ref{C:ergodic}, and Condition \ref{C:center}, the infimization problem \eqref{E:ordLocalRate} has the explicit solution
\begin{equation*}
  L^o(x, \eta, \beta) = \frac{1}{2} (\beta - \kappa(x, \eta))^\T q^{-1}(x) (\beta - \kappa(x, \eta))
\end{equation*}
where $\kappa(x, \eta)$ and $q(x)$ are given by \eqref{E:rDef} and \eqref{E:qDef}. Furthermore, with $\alpha_1(x, y)$, $\alpha_2(x, y)$ given by \eqref{E:alphaDef}, the control $v(y) = (v_1(y), v_2(y))$ defined by
\begin{align*}
  v_1(y) &= \alpha_1(x,y)^\T q^{-1}(x)(\beta - \kappa(x, \eta)) \\
  v_2(y) &= \alpha_2(x,y)^\T q^{-1}(x)(\beta - \kappa(x, \eta)) \notag
\end{align*}
attains the infimum in the variational problem \eqref{E:ordLocalRate}.
\end{theorem}

\begin{proof}
Observe that for any $v \in \mathcal{A}_{x, \eta, \beta}^o$,
\begin{equation*}
  \int_\mathcal{Y} \lvert v(y) \rvert^2 \, \mu_{1,x}(dy) \ge (\beta - \kappa(x, \eta))^\T q^{-1}(x)(\beta - \kappa(x, \eta)).
\end{equation*}
This follows because any $v \in \mathcal{A}_{x, \eta, \beta}^o$ satisfies
\begin{align*}
  \beta &= \int_\mathcal{Y} \theta_1(x, \eta, y , v_1(y), v_2(y)) \,\mu_{1,x}(dy) \\
\notag  &= \kappa(x, \eta) + \int_\mathcal{Y} \left[ \sigma(x,y) v_1(y) + \big( \nabla_y \chi(x,y) \big) (\tau_1(x,y)v_1(y) + \tau_2(x,y) v_2(y)) \right] \mu_{1,x}(dy) \\
\notag  &= \kappa(x, \eta) + \int_\mathcal{Y} \big( \alpha_1(x,y) v_1(y) + \alpha_2(x,y) v_2(y) \big) \mu_{1,x}(dy).
\end{align*}
Then treating $x$ and $\eta$ as parameters and applying Lemma 5.1 from \cite{DS12} to the relation above, we get the claim. Next, observe that by choosing (with $x$ and $\eta$ treated as parameters)
\begin{align*}
  v_1(y) &= \alpha_1(x,y)^\T q^{-1}(x)(\beta - \kappa(x, \eta)) \\
  v_2(y) &= \alpha_2(x,y)^\T q^{-1}(x)(\beta - \kappa(x, \eta)) \notag
\end{align*}
we have
\begin{equation*}
  \int_\mathcal{Y} \lvert v(y) \rvert^2 \,\mu_{1,x}(dy) = (\beta - \kappa(x, \eta))^\T q^{-1}(x) (\beta - \kappa(x, \eta)).
\end{equation*}
This completes the proof of the theorem.
\end{proof}

Now we can prove the Laplace principle upper bound. We must show that for all bounded, continuous functions $a$ mapping $\mathcal{C}([0,1]; \mathbb{R}^n)$ into $\mathbb{R}$
\begin{equation*}
  \limsup_{\varepsilon \downarrow 0} - \frac{1}{h^2(\varepsilon)} \log \E \left[ \exp \left\{ -h^2(\varepsilon) a(\eta^\varepsilon) \right\} \right] \le \inf_{\xi \in \mathcal{C}([0,1]; \mathbb{R}^n)} [S(\xi) + a(\xi)].
\end{equation*}
Let $\zeta > 0$ be given and consider $\psi \in \mathcal{C}([0,1];\mathbb{R}^n)$ with $\psi_0 = 0$ such that
\begin{equation*}
  S(\psi) + a(\psi) \le \inf_{\xi \in \mathcal{C}([0,1]: \mathbb{R}^n)} [S(\xi) + a(\xi)] + \zeta < \infty.
\end{equation*}
Since $a$ is bounded, this implies that $S(\psi) < \infty$, and thus $\psi$ is absolutely continuous. Theorem \ref{T:controlRepresentation} shows that $L^o(x, \eta, \beta)$ is continuous and finite at each $(x, \eta, \beta) \in \mathbb{R}^{3n}$. By a mollification argument we can assume that $\dot{\psi}$ is piecewise continuous, see Section 6.5 in \cite{DE97}. Given this $\psi$ define
\begin{align*}
  \bar{u}_1(t, x, \eta, y) &= \alpha_1(x,y)^\T q^{-1}(x)(\dot{\psi}_t - \kappa(x, \eta)) \\
  \bar{u}_2(t, x, \eta, y) &= \alpha_2(x,y)^\T q^{-1}(x)(\dot{\psi}_t - \kappa(x, \eta)) \notag
\end{align*}
with $\alpha_1$ and $\alpha_2$ defined as in Theorem \ref{T:controlRepresentation}. Define a control in feedback form by
\begin{equation*}
  \bar{u}^\varepsilon(t) = (\bar{u}_1(t), \bar{u}_2(t)) =  \left( \bar{u}_1(t, \bar{X}_t, \eta_t^\varepsilon, Y_t^\varepsilon),  \bar{u}_2(t, \bar{X}_t, \eta_t^\varepsilon, Y_t^\varepsilon)\right).
\end{equation*}

Then ${\eta}^{\varepsilon,\bar{u}^\varepsilon} \to \bar{\eta}$ in distribution, where w.p.1
\begin{align*}
  &\bar{\eta}_t = \int_0^t \kappa(\bar{X}_s, \bar{\eta}_s) \,ds
  + \int_0^t \left[ \int_\mathcal{Y} \left[ \left( \sigma(\bar{X}_s, y) + \big( \nabla_y \chi(\bar{X}_s, y) \big) \tau_1(\bar{X}_s, y) \right) \bar{u}_1(s) \right. \right. \\
  &\quad \left.\vphantom{\int_\mathcal{Y}} \left. + \big( \nabla_y \chi(\bar{X}_s, y) \big) \tau_2(\bar{X}_s, y)\bar{u}_2(s) \right] \mu_{1, \bar{X}_s}(dy) \right] ds \notag \\
  &= \int_0^t \kappa(\bar{X}_s, \bar{\eta}_s) \,ds + \int_0^t \left[ \int_\mathcal{Y} \left[ \alpha_1 \alpha_1^\T(\bar{X}_s,y) + \alpha_2 \alpha_2^\T(\bar{X}_s,y) \right] \mu_{1, \bar{X}_s}(dy) \right] q^{-1}(\bar{X}_s)(\dot{\psi}_s - \kappa(\bar{X}_s, \bar{\eta}_s)) \,ds \notag \\
  &= \int_0^t \kappa(\bar{X}_s, \bar{\eta}_s) \,ds + \int_0^t q(\bar{X}_s) q^{-1}(\bar{X}_s)(\dot{\psi}_s - \kappa(\bar{X}_s, \bar{\eta}_s)) \,ds \notag \\
  &= \int_0^t \dot{\psi}_s \,ds = \psi_t. \notag
\end{align*}

The cost satisfies
\begin{equation*}
  \E \left( \frac{1}{2} \int_0^1 \lvert \bar{u}_s^\varepsilon \rvert^2 \,ds - \frac{1}{2} \int_0^1 \int_\mathcal{Y} \lvert \bar{u}(s, \bar{X}_s, \bar{\eta}_s, y) \rvert^2 \,\mu_{1, \bar{X}_s}(dy) \,ds \right)^2 \to 0 \text{ as } \varepsilon \downarrow 0.
\end{equation*}

Theorem \ref{T:controlRepresentation} implies that
\begin{equation*}
  \E \frac{1}{2} \int_0^1 \int_\mathcal{Y} \lvert \bar{u}(s, \bar{X}_s, \bar{\eta}_s, y) \rvert^2 \,\mu_{1, \bar{X}_s}(dy) \,ds = \E S(\bar{\eta}) = S(\psi).
\end{equation*}

Then we obtain
\begin{align*}
  &\limsup_{\varepsilon \downarrow 0}  - \frac{1}{h^2(\varepsilon)} \log \E \left[ \exp \left\{ -h^2(\varepsilon) a(\eta^\varepsilon) \right\} \right] = \limsup_{\varepsilon \downarrow 0} \inf_{u \in \mathcal{A}} \E \left[ \frac{1}{2} \int_0^1 \lvert u(t) \rvert^2 \,dt + a(\eta^{\varepsilon, u}) \right] \\
  &\qquad \le \limsup_{\varepsilon \downarrow 0}  \E \left[ \frac{1}{2} \int_0^1 \lvert \bar{u}^\varepsilon(t) \rvert^2 \,dt + a(\eta^{\varepsilon, \bar{u}^\varepsilon}) \right] \notag \\
  &\qquad = \E \left[ \frac{1}{2} \int_0^1 \int_\mathcal{Y} \lvert \bar{u}(s, \bar{X}_s, \bar{\eta}_s, y) \rvert^2 \,\mu_{1, \bar{X}_s}(dy) \,ds + a(\bar{\eta}) \right] \notag \\
  &\qquad = [S(\psi) + a(\psi)] \notag \\
  &\qquad \le \inf_{\xi \in \mathcal{C}([0,1]; \mathbb{R}^n)} [S(\xi) + a(\xi)] + \zeta. \notag
\end{align*}

Since $\zeta >0$ is arbitrary, the upper bound is proved. Furthermore, we have an explicit representation formula for the action functional, given by
\begin{equation*}
  S(\xi) = \frac{1}{2} \int_0^1 \left(\dot{\xi}_s - \kappa \left(\bar{X}_s, \xi_{s} \right) \right)^\T q^{-1}(\bar{X}_s) \left(\dot{\xi}_s - \kappa \left(\bar{X}_s, \xi_{s} \right) \right) \,ds
\end{equation*}
if $\xi \in \mathcal{C}([0, 1]; \mathbb{R}^n)$ is absolutely continuous, and $\infty$ otherwise.

\section{Comments on the Proofs for Regime 2}

The structure of the proof for Regime 2 is identical to that of Regime 1, after replacing $\lambda_1, \theta_1, \opL_1, \Phi_1,$ and $\mu_1$ by $\lambda_2, \theta_2, \opL_2, \Phi_2,$ and $\mu_2$ respectively. Hence we do not repeat it here. For example in Regime 2, applying the It\^{o} formula to $\Phi_2(X_t^{\varepsilon, u^\varepsilon}, Y_t^{\varepsilon, u^\varepsilon})$ and some term rearranging shows that
\begin{align*}
  \eta_t^{\varepsilon, u^\varepsilon} &= \int_0^t \left[\vphantom{\frac{j_2}{2} \tau_1^\T} j_2 b(\bar{X}_s, Y_s^{\varepsilon, u^\varepsilon} ) + \big( \nabla_y \Phi_2(\bar{X}_s, Y_s^{\varepsilon, u^\varepsilon} ) \big) \left[ \tau_1(\bar{X}_s, Y_s^{\varepsilon, u^\varepsilon} ) u_{1}^\varepsilon(s) + \tau_2(\bar{X}_s, Y_s^{\varepsilon, u^\varepsilon} ) u_{2}^\varepsilon(s) \right] \right.\\
  \notag &\left. + \big( \nabla_x \bar{\lambda}_2(\bar{X}_s) \big) \eta_s^{\varepsilon, u^\varepsilon} + \sigma(\bar{X}_s, Y_s^{\varepsilon, u^\varepsilon} ) u_1^\varepsilon(s) + j_2 \big( \nabla_y \Phi_2( X_s^{\varepsilon, u^\varepsilon}, Y_s^{\varepsilon, u^\varepsilon} ) \big) f( X_s^{\varepsilon, u^\varepsilon}, Y_s^{\varepsilon, u^\varepsilon} ) \right. \\
  \notag &\left. + \frac{j_2}{2} \left( \left(\tau_1\tau_1^\T + \tau_2\tau_2^\T \right)( X_s^{\varepsilon, u^\varepsilon}, Y_s^{\varepsilon, u^\varepsilon} ) : \nabla_y\nabla_y \Phi_2( X_s^{\varepsilon, u^\varepsilon}, Y_s^{\varepsilon, u^\varepsilon} ) \right) \right] \,ds + R^\varepsilon,
\end{align*}
where $R^\varepsilon$ contains the  additional terms which go to zero as $\varepsilon\downarrow 0$. However, the necessary statements that were needed for Regime 1 and which are proved in Appendix B do need some special care. We address these in Appendix C.

\appendix

\section{Regularity results}

The following theorem collects results from \cite{PV01} and \cite{PV03} that are used in this paper.
\begin{theorem} \label{T:regularity}
Let Conditions \ref{C:growth} and \ref{C:ergodic} be satisfied. In Regime $i = 1, 2$ we have that,
\begin{enumerate}[(i)]
\item There exists a unique invariant measure $\mu_{i, x}(dy)$ associated with the operator $\opL_{i, x}$. For all $x \in \mathbb{R}^n$ and $q \in \mathbb{N}$,
\begin{equation*}
    \int_\mathcal{Y} \lvert y \rvert^q \,\mu_{i, x}(dy) < \infty.
\end{equation*}
Moreover, $\mu_{i, x}$ has a density which is twice differentiable in $x$.

\item Assume that $G(x, y) \in \mathcal{C}^{2, \alpha}(\mathbb{R}^n\times \mathcal{Y})$. Then
\begin{equation*}
    \bar{G}(x) = \int_\mathcal{Y} G(x, y) \,\mu_{i, x}(dy)
\end{equation*}
is twice differentiable in $x$.

\item Assume that $F(x, y) \in \mathcal{C}^{2, \alpha}(\mathbb{R}^n\times \mathcal{Y})$,
\begin{equation*}
    \int_\mathcal{Y} F(x, y) \,\mu_{i, x}(dy) = 0,
\end{equation*}
and that for some positive constants $K$ and $q_{F}$,
\begin{equation*}
   \lvert F(x, y) \rvert + \lVert \nabla_x F(x, y) \rVert + \lVert \nabla_x \nabla_x F(x, y) \rVert \le K (1 + \lvert y \rvert^{q_{F}}).
\end{equation*}
Then there is a unique solution from the class of functions which grow at most polynomially in $\lvert y \rvert$  to
\begin{equation*}
    \opL_{i, x} u(x, y) = - F(x, y), \quad \int_\mathcal{Y} u(x, y) \,\mu_{i, x}(dy) = 0.
\end{equation*}
Moreover, the solution satisfies $u(\cdot, y) \in \mathcal{C}^2$ for every $y \in \mathcal{Y}$, $\nabla_x \nabla_x u \in \mathcal{C}(\mathbb{R}^n\times\mathcal{Y})$, and there exist positive constants $K'$ that change from line to line such that
\begin{align*}
\lvert u(x, y) \rvert &\le K' (1 + \lvert y \rvert )^{(q_{F}+1-r)^{+}},\nonumber\\
 \lVert \nabla_y u(x, y) \rVert &\le K' (1 + \lvert y \rvert ^{(q_{F}+1-r)^{+}}+\lvert y \rvert ^{q_{F}})\nonumber\\
\lVert \nabla_x u(x, y) \rVert &\le K' (1 + \lvert y \rvert ^{(q_{F}+1-r)^{+}}+\lvert y \rvert ^{(q_{F}+2(1-r))^{+}}),\nonumber\\
 \lVert \nabla_x \nabla_x u(x, y) \rVert &\le K' (1 + \lvert y \rvert ^{(q_{F}+1-r)^{+}}+\lvert y \rvert ^{(q_{F}+2(1-r))^{+}}+\lvert y \rvert ^{(q_{F}+3(1-r))^{+}} )
\end{align*}
where $r$ is as defined in Condition \ref{C:ergodic}.
\end{enumerate}
\end{theorem}
\begin{proof}
(i) Conditions \ref{C:growth} and \ref{C:ergodic} imply that $\opL_{i, x}$ satisfies the conditions for Proposition 1 in \cite{PV01} and Theorem 1 in \cite{PV03}. The first statement is due to Proposition 1 in \cite{PV01} and the second statement is due to Theorem 1 in \cite{PV03}.

(ii) Conditions \ref{C:growth} and \ref{C:ergodic} and the condition on $G$ imply that Theorem 2 in \cite{PV03} holds, so (ii) holds.

(iii) Conditions \ref{C:growth} and \ref{C:ergodic} and the conditions on $F$ imply that Theorem 3 in \cite{PV03} holds, which implies the existence and smoothness of $u$. The corresponding growth conditions follow from Theorem 2 of \cite{PV01}, appropriately translated to our case. Notice that Theorem 2 of \cite{PV01} has a statement for the growth only for the solution and its $y-$derivative. The statements for the $x-$derivatives follow, for example, by differentiating the equation and re-applying Theorem 2 of \cite{PV01} to the new equation.
\end{proof}

\section{Lemmas for Regime 1}

\begin{lemma} \label{L:uBound}
Assume Conditions \ref{C:growth}, \ref{C:ergodic} and \ref{C:StrongSolution}. Let $(X_t^{\varepsilon, u^\varepsilon}, Y_t^{\varepsilon, u^\varepsilon})$ be the strong solution to \eqref{E:controlledSDE}. Then the infimum of the representation in \eqref{E:controlRepresentation} can be taken over all controls such that
\begin{equation*}
    \int_0^1  \lvert u^\varepsilon(s)  \rvert^2 \,ds < N, \text{ almost surely},
\end{equation*}
where the constant $N$ does not depend on $\varepsilon$ or $\delta$.
\end{lemma}

\begin{proof}
The proof is standard, but we recall it here for the readers convenience. Without loss of generality, we can consider a function $a(x)$ that is bounded and uniformly Lipschitz continuous in $\mathcal{C}([0, 1]; \mathbb{R}^n)$.
Namely, there exists a constant $L_a$ such that
\begin{equation*}
    \lvert a(x) - a(y) \rvert \le L_a \lVert x - y \rVert
\end{equation*}
(where $\lVert \cdot \rVert$ is the supremum norm) and $\lVert a \rVert_\infty = \sup_{x \in \mathcal{C}([0, 1]; \mathbb{R}^n)} \lvert a(x) \rvert < \infty$.

Fix $\zeta > 0$. There exists a family of controls $\{u^\varepsilon, \varepsilon >0 \}$ in $\mathcal{A}$ such that for every $\varepsilon > 0$,
\begin{equation*}
    - \frac{1}{h^2(\varepsilon)} \log \E \left[ \exp\{-h^2(\varepsilon)a(\eta^\varepsilon)\}\right] \ge \E \left[ \frac{1}{2} \int_0^1 \lvert u^\varepsilon(s)\rvert^2 \,ds + a(\eta^{\varepsilon,u^\varepsilon}) \right] - \zeta.
\end{equation*}

Then each control $u^\varepsilon$ satisfies
\begin{equation*}
    \sup_{\varepsilon > 0} \E \left[ \frac{1}{2} \int_0^1 \lvert u^\varepsilon(s) \rvert^2 \,ds \right] \le 2 \lVert a \rVert_\infty + \zeta.
\end{equation*}

By the proof of Theorem 4.4 in \cite{BD00}, it is enough to assume that for given $\zeta > 0$ the controls satisfy the bound
\begin{equation*}
    \int_0^1 \lvert u^\varepsilon(s) \rvert^2 \,ds < N
\text{ where }    N \ge \frac{4 \lVert a \rVert_\infty ( 4 \lVert a \rVert_\infty + \zeta )}{\zeta}.
\end{equation*}
\end{proof}

\begin{lemma} \label{L:Ygrowth}
Let Conditions \ref{C:growth}, \ref{C:ergodic}, \ref{C:Tightness}, and \ref{C:StrongSolution} be satisfied. For  $N\in\mathbb{N}$, let $u^{\varepsilon} \in \mathcal{A}$  such that almost surely
\begin{equation*}
    \sup_{\varepsilon > 0} \int_0^1 \lvert u^{\varepsilon}(s) \rvert^2 \,ds < N.
\end{equation*}
 Then, with $r>0$ from Condition \ref{C:ergodic} we have for any $T\leq 1$ there exist $\varepsilon_{0}>0$ small enough such that
 \begin{equation*}
    \sup_{\varepsilon\in(0,\varepsilon_{0})}\E  \int_{0}^{T} | Y_s^{\varepsilon, u^{\varepsilon}}|^{2r} \,ds  \le K(N,T,r)
\end{equation*}
for some finite constant $K(N,T,r)$ that may depend on $(N,T,r)$, but not on $\varepsilon,\delta(\varepsilon)$.
\end{lemma}
\begin{proof}
By the Markov property it is enough to check what happens when the process $Y^{\varepsilon, u^{\varepsilon}}$ is outside a compact subset of $\mathcal{Y}$. For this purpose, with $R<\infty$ to be chosen, let us define
\[
\tau_{R}=\inf \left\{t>0: |Y^{\varepsilon, u^{\varepsilon}}_{t}|<R\right\}
\]
and assume that the initial condition is such that $|y_{0}|>R$. For notational convenience, we will write $(X,Y)$ instead of $(X^{\varepsilon,u^{\varepsilon}},Y^{\varepsilon,u^{\varepsilon}})$. Without loss of generality and for exposition purposes we shall set $g=0$ (since it is assumed to be bounded) and $\tau_{1}=0$ (the argument is exactly the same if both $\tau_{1}$ and $\tau_{2}$ are non-zero). By Condition \ref{C:ergodic} we have that uniformly in both $x$ and $y$ and for any constant $\beta$
\[
\frac{\beta-2}{2}\frac{\left<\tau_{2}\tau^{T}_{2}(x,y)y, y\right>}{|y|^{2}}+\frac{1}{2}
\textrm{Tr}(\tau_{2}\tau_{2}^{T})(x,y)\leq \rho,
\]
for some fixed constant $\rho>0$. Hence, considering $t\leq T$, the It\^{o} formula gives for $\beta>0$ (to be chosen)
\begin{align}
&\E|Y_{t\wedge \tau_{R}}|^{\beta}=|y_{0}|^{\beta}+\beta\frac{\varepsilon}{\delta^{2}}\E\int_{0}^{t\wedge\tau_{R}}|Y_{s}|^{\beta-2}
\left(\vphantom{\frac{\beta}{2}} \left<Y_{s},f(X_{s},Y_{s})\right>+\right.\nonumber\\
&\quad\left.+\frac{\beta-2}{2}\frac{\left<\tau_{2}\tau^{T}_{2}(X_{s},Y_{s})Y_{s}, Y_{s}\right>}{|Y_{s}|^{2}}+\frac{1}{2}
\textrm{Tr}(\tau_{2}\tau_{2}^{T})(X_{s},Y_{s})\right)ds\nonumber\\
&\qquad+\frac{\sqrt{\varepsilon} h(\varepsilon)}{\delta}\beta \E\int_{0}^{t\wedge\tau_{R}}|Y_{s}|^{\beta-2}
\left<Y_{s},\tau_{2}(X_{s},Y_{s})u^{\varepsilon}_{2}(s)\right>ds\nonumber\\
&\leq |y_{0}|^{\beta}+\beta\frac{\varepsilon}{\delta^{2}}\E\int_{0}^{t\wedge\tau_{R}}|Y_{s}|^{\beta-2}
\left(-\Gamma |Y_{s}|^{r+1}+\rho\right)ds+\nonumber\\
&+\frac{\sqrt{\varepsilon} h(\varepsilon)}{2\delta}\beta \E\int_{0}^{t\wedge\tau_{R}}\left(|Y_{s}|^{2\beta-2}+\left\Vert \tau_{2}\right\Vert^{2} |u^{\varepsilon}(s)|^{2}
\right)ds\nonumber
\end{align}
where Condition \ref{C:Tightness} was used.

Choosing now $R$ large enough such that $R^{1+r}>\frac{2\rho}{\Gamma}$ and recalling that
$\sup_{\varepsilon\in(0,1)}\E\int_{0}^{T}|u^{\varepsilon}(s)|^{2}ds\leq N$, we can continue the last inequality as follows
\begin{align}
\E|Y_{t\wedge \tau_{R}}|^{\beta}
&\leq |y_{0}|^{\beta}-\frac{\beta\Gamma}{2}\frac{\varepsilon}{\delta^{2}}\E\int_{0}^{t\wedge\tau_{R}}|Y_{s}|^{\beta+r-1}
ds+\frac{\sqrt{\varepsilon} h(\varepsilon)}{2\delta}\beta \left\Vert \tau_{2}\right\Vert^{2}  N \nonumber\\
&\qquad+\frac{\sqrt{\varepsilon} h(\varepsilon)}{2\delta}\beta  \E\int_{0}^{t\wedge\tau_{R}}|Y_{s}|^{2\beta-2}ds\nonumber
\end{align}
Choosing now $\beta\leq r+1$, we obtain
\begin{align}
\E|Y_{t\wedge \tau_{R}}|^{\beta}
&\leq |y_{0}|^{\beta}-\frac{\beta\Gamma}{2}\frac{\varepsilon}{\delta^{2}}\E\int_{0}^{t\wedge\tau_{R}}|Y_{s}|^{\beta+r-1}
ds+\frac{\sqrt{\varepsilon} h(\varepsilon)}{2\delta}\beta \left\Vert \tau_{2}\right\Vert^{2}  N \nonumber\\
&\qquad+\frac{\sqrt{\varepsilon} h(\varepsilon)}{2\delta}\beta \E\int_{0}^{t\wedge\tau_{R}}|Y_{s}|^{\beta+r-1}ds\nonumber
\end{align}
Choosing next $\varepsilon,\delta$ sufficiently small such that $\frac{\delta}{\varepsilon}\sqrt{\varepsilon} h(\varepsilon)<\Gamma/4$ we obtain
\begin{align}
\E|Y_{t\wedge \tau_{R}}|^{\beta}
&\leq |y_{0}|^{\beta}-\frac{\beta\Gamma}{4}\frac{\varepsilon}{\delta^{2}}\E\int_{0}^{t\wedge\tau_{R}}|Y_{s}|^{\beta+r-1}
ds+\frac{\sqrt{\varepsilon} h(\varepsilon)}{2\delta}\beta \left\Vert \tau_{2}\right\Vert^{2}  N \nonumber
\end{align}
which then gives by comparison
\begin{align}
\E\int_{0}^{t\wedge\tau_{R}}|Y_{s}|^{\beta+r-1}
ds
&\leq \frac{4}{\beta\Gamma}\frac{\delta^{2}}{\varepsilon}|y_{0}|^{\beta}+\frac{2}{\Gamma}\frac{\delta h(\varepsilon)}{\sqrt{\varepsilon}} \left\Vert \tau_{2}\right\Vert^{2}  N. \nonumber
\end{align}

Since $\beta\leq r+1$, we choose $\beta=r+1$, which concludes the proof of the lemma.
\end{proof}

\begin{lemma} \label{L:productBound}
Let Conditions \ref{C:growth} and \ref{C:ergodic} be satisfied. Let $N\in\mathbb{N}$ be finite and $u^\varepsilon = (u_1^\varepsilon, u_2^\varepsilon) \in \mathcal{A}$ such that almost surely
\begin{equation*}
    \sup_{\varepsilon > 0} \int_0^1  \lvert u^\varepsilon(s) \rvert^2  \,ds < N.
\end{equation*}
 Let $A(x, y)$ and $B(x,y)$ be matrix-valued functions and $K$, $\theta\in(0,r)$, where $r$ is as defined in Condition \ref{C:ergodic}, be constants such that each components $A_{ij}$ and $B_{ij}$ satisfy
\begin{equation*}
    |A_{ij}(x, y)| \le K (1 + \lvert y \rvert^\theta), \text{ and }|B_{ij}(x, y)| \le K (1 + \lvert y \rvert^{2\theta}).
\end{equation*}
Then for $\alpha \in \{1, 2\}$:
\begin{enumerate}[(i)]
\item For any $p\in(1,r/\theta]$, there exists a $C<\infty$ such that for fixed $\rho > 0$ and for all $0 \le t_1 < t_1+\rho \le 1$,
\begin{equation*}
    \E \sup_{\substack{0 \le t_1 < t_2 \le 1 \\ \lvert t_2 - t_1 \rvert < \rho}}\left\lvert \int_{t_1}^{t_{2}} A(X_s^{\varepsilon, u^\varepsilon}, Y_s^{\varepsilon, u^\varepsilon}) u_\alpha^\varepsilon(s) \,ds \right\rvert^{2p} \le C \lvert \rho \rvert^{r/\theta-1}.
\end{equation*}
and
\begin{equation*}
    \E \sup_{\substack{0 \le t_1 < t_2 \le 1 \\ \lvert t_2 - t_1 \rvert < \rho}}\left\lvert \int_{t_1}^{t_{2}} B(X_s^{\varepsilon, u^\varepsilon}, Y_s^{\varepsilon, u^\varepsilon})  \,ds \right\rvert^{p} \le C \lvert \rho \rvert^{r/\theta-1}.
\end{equation*}
\item For all $\zeta > 0$
\begin{equation*}
    \lim_{\rho \downarrow 0} \limsup_{\varepsilon \downarrow 0} \p \left[ \sup_{\substack{0 \le t_1 < t_2 \le 1 \\ \lvert t_2 - t_1 \rvert < \rho}} \left\lvert \int_{t_1}^{t_2} A(X_s^{\varepsilon, u^\varepsilon}, Y_s^{\varepsilon, u^\varepsilon}) u_\alpha^\varepsilon(s) \,ds \right\rvert > \zeta \right] = 0.
\end{equation*}
and
\begin{equation*}
    \lim_{\rho \downarrow 0} \limsup_{\varepsilon \downarrow 0} \p \left[ \sup_{\substack{0 \le t_1 < t_2 \le 1 \\ \lvert t_2 - t_1 \rvert < \rho}} \left\lvert \int_{t_1}^{t_2} B(X_s^{\varepsilon, u^\varepsilon}, Y_s^{\varepsilon, u^\varepsilon})  \,ds \right\rvert > \zeta \right] = 0.
\end{equation*}
\end{enumerate}
\end{lemma}

\begin{proof}
We shall only prove the statement for $A(x,y)$, as the proof for $B(x,y)$ is the same but simpler. Applying H\"{o}lder inequality with  $1/m+1/q=1$ and $q=p>1$ gives
\begin{align*}
&\E \sup_{\substack{0 \le t_1 < t_2 \le 1 \\ \lvert t_2 - t_1 \rvert < \rho}}\left\lvert \int_{t_1}^{t_2} A(X_s^{\varepsilon, u^\varepsilon}, Y_s^{\varepsilon, u^\varepsilon}) u_\alpha^\varepsilon(s) \,ds \right\rvert^{2p}\nonumber\\
 &\quad \le \E \sup_{\substack{0 \le t_1 < t_2 \le 1 \\ \lvert t_2 - t_1 \rvert < \rho}}\left( \int_{t_1}^{t_2} \left|A(X_s^{\varepsilon, u^\varepsilon}, Y_s^{\varepsilon, u^\varepsilon})\right|^{2}ds\right)^{p} \left( \int_{t_1}^{t_2} \left|u_\alpha^\varepsilon(s)\right|^{2} ds \right)^{p}\nonumber\\
&\quad \le N^{p} \E \sup_{\substack{0 \le t_1 < t_2 \le 1 \\ \lvert t_2 - t_1 \rvert < \rho}}\left( \int_{t_1}^{t_2} \left|A(X_s^{\varepsilon, u^\varepsilon}, Y_s^{\varepsilon, u^\varepsilon})\right|^{2}ds\right)^{p} \nonumber\\
&\quad \le N^{p} \rho^{p/m}\E \sup_{\substack{0 \le t_1 < t_2 \le 1 \\ \lvert t_2 - t_1 \rvert < \rho}}\left( \int_{t_1}^{t_2} \left|A(X_s^{\varepsilon, u^\varepsilon}, Y_s^{\varepsilon, u^\varepsilon})\right|^{2q}ds\right)^{p/q} \nonumber\\
&\quad \le N^{p} \rho^{p/m}\E \sup_{\substack{0 \le t_1 < t_2 \le 1 \\ \lvert t_2 - t_1 \rvert < \rho}}\left( \int_{t_1}^{t_2} \left|A(X_s^{\varepsilon, u^\varepsilon}, Y_s^{\varepsilon, u^\varepsilon})\right|^{2p}ds\right) \nonumber\\
&\quad \le N^{p} \rho^{p/m}\E \sup_{\substack{0 \le t_1 < t_2 \le 1 \\ \lvert t_2 - t_1 \rvert < \rho}} \int_{t_1}^{t_2} \left(1+\left|Y_s^{\varepsilon, u^\varepsilon}\right|^{2\theta p}\right)ds \nonumber\\
&\quad \le N^{p} \rho^{p-1}\E  \int_{0}^{1} \left(1+\left|Y_s^{\varepsilon, u^\varepsilon}\right|^{2r}\right)ds \nonumber
\end{align*}
and the result follows by Lemma \ref{L:Ygrowth} and by the choice of $p$.

The second claim follows from the first statement and Markov's inequality.
\end{proof}

\begin{lemma}\label{L:coefs}
Assume Conditions \ref{C:growth}, \ref{C:ergodic}, \ref{C:Tightness} and \ref{C:center} and define the function $\chi(x, y)$ by \eqref{E:cell} and the processes $X^{\varepsilon, u^\varepsilon}$ and $Y^{\varepsilon, u^\varepsilon}$ by \eqref{E:controlledSDE}. Let $N<\infty$ such that almost surely
\begin{equation*}
  \sup_{\varepsilon > 0}  \int_0^1 \lvert u^\varepsilon(s) \rvert^2 \,ds < N.
\end{equation*}

We have
\begin{enumerate}[(i)]
\item
\begin{align*}
  &\E \sup_{t \in [0, 1]} \left\lvert \int_0^t \frac{ \frac{\varepsilon}{\delta} b(X_s^{\varepsilon,u^\varepsilon}, Y_s^{\varepsilon,u^\varepsilon}) + c(X_s^{\varepsilon,u^\varepsilon}, Y_s^{\varepsilon,u^\varepsilon}) - \lambda_1(X_s^{\varepsilon,u^\varepsilon}, Y_s^{\varepsilon,u^\varepsilon}) }{\sqrt{\varepsilon}h(\varepsilon)} \,ds \right. \\
  \notag & \left. - \int_0^t \big( \nabla_y \chi(X_s^{\varepsilon,u^\varepsilon}, Y_s^{\varepsilon,u^\varepsilon}) \big) \left[ \tau_1(X_s^{\varepsilon,u^\varepsilon}, Y_s^{\varepsilon,u^\varepsilon}) u_1^\varepsilon(s) + \tau_2(X_s^{\varepsilon,u^\varepsilon}, Y_s^{\varepsilon,u^\varepsilon}) u_2^\varepsilon(s) \right] \,ds \right\rvert^2 \\
  \notag & < ( \delta C_1)^2 + \left(\frac{1}{h(\varepsilon)} C_2 \right)^2   + o \left( \delta^2 + \frac{1}{h^2(\varepsilon)} \right)
\end{align*}

\item For every $\zeta > 0$,
\begin{multline} \label{E:probLim}
  \lim_{\rho \downarrow 0} \limsup_{\varepsilon \downarrow 0} \p \left[ \sup_{\substack{0 \le t_1 < t_2 \le 1 \\ \lvert t_2 - t_1 \rvert < \rho}} \left\lvert \int_{t_1}^{t_2} \frac{1}{\sqrt{\varepsilon} h(\varepsilon)} \left( \frac{\varepsilon}{\delta}b(X_s^{\varepsilon, u^\varepsilon}, Y_s^{\varepsilon, u^\varepsilon} ) \right. \right. \right. \\
  \left.\vphantom{\sup_{\substack{0 \le t_1 < t_2 \le 1 \\ \lvert t_2 - t_1 \rvert < \rho}}} \left.\vphantom{\int_{t_1}^{t_2}} \left.\vphantom{\frac{\varepsilon}{\delta}} + c(X_s^{\varepsilon, u^\varepsilon}, Y_s^{\varepsilon, u^\varepsilon} ) - \lambda_1( X_s^{\varepsilon, u^\varepsilon}, Y_s^{\varepsilon, u^\varepsilon}) \right) \,ds \right\rvert > \zeta \right] = 0 .
\end{multline}

\item There exists an $M > 0$ such that for all sufficiently small $\varepsilon$,
\begin{equation} \label{E:expBound}
  \E \sup_{t\in[0,1]}\left\lvert \int_0^t \frac{ \frac{\varepsilon}{\delta}b(X_s^{\varepsilon, u^\varepsilon}, Y_s^{\varepsilon, u^\varepsilon} ) + c(X_s^{\varepsilon, u^\varepsilon}, Y_s^{\varepsilon, u^\varepsilon} ) - \lambda_1( X_s^{\varepsilon, u^\varepsilon}, Y_s^{\varepsilon, u^\varepsilon}) }{\sqrt{\varepsilon} h(\varepsilon)} \,ds \right\rvert^2 < M .
\end{equation}
\end{enumerate}
\end{lemma}

\begin{proof}
(i). Applying the It\^{o} formula to $\chi(X_t^{\varepsilon, u^\varepsilon}, Y_t^{\varepsilon, u^\varepsilon})$ and rearranging gives
\begin{align} \label{E:coefsExpand}
  &\frac{1}{\sqrt{\varepsilon}h(\varepsilon)} \int_0^t \left[ \frac{\varepsilon}{\delta} b(X_s^{\varepsilon,u^\varepsilon}, Y_s^{\varepsilon,u^\varepsilon}) + c(X_s^{\varepsilon,u^\varepsilon}, Y_s^{\varepsilon,u^\varepsilon}) - \lambda_1(X_s^{\varepsilon,u^\varepsilon}, Y_s^{\varepsilon,u^\varepsilon}) \right] \,ds \\
  \notag &= - \frac{\delta}{\sqrt{\varepsilon}h(\varepsilon)} \left( \chi(X_t^{\varepsilon,u^\varepsilon}, Y_t^{\varepsilon,u^\varepsilon}) - \chi(x_0, y_0) \right) \\
  \notag &+ \frac{\delta}{\sqrt{\varepsilon}h(\varepsilon)} \int_0^t \big( \nabla_x \chi(X_s^{\varepsilon,u^\varepsilon}, Y_s^{\varepsilon,u^\varepsilon}) \big) \left[ \frac{\varepsilon}{\delta} b(X_s^{\varepsilon,u^\varepsilon}, Y_s^{\varepsilon,u^\varepsilon}) + c(X_s^{\varepsilon,u^\varepsilon}, Y_s^{\varepsilon,u^\varepsilon}) \right] \,ds \\
  \notag &+ \delta \int_0^t \big( \nabla_x \chi(X_s^{\varepsilon,u^\varepsilon}, Y_s^{\varepsilon,u^\varepsilon}) \big) \sigma(X_s^{\varepsilon,u^\varepsilon}, Y_s^{\varepsilon,u^\varepsilon}) u_1^\varepsilon(s) \,ds \\
  \notag &+ \int_0^t \big( \nabla_y \chi(X_s^{\varepsilon,u^\varepsilon}, Y_s^{\varepsilon,u^\varepsilon}) \big) \left[ \tau_1(X_s^{\varepsilon,u^\varepsilon}, Y_s^{\varepsilon,u^\varepsilon}) u_1^\varepsilon(s) + \tau_2(X_s^{\varepsilon,u^\varepsilon}, Y_s^{\varepsilon,u^\varepsilon}) u_2^\varepsilon(s) \right] \,ds \\
  \notag &+ \frac{\delta \sqrt{\varepsilon}}{2h(\varepsilon)} \int_0^t \sigma \sigma^\T (X_s^{\varepsilon,u^\varepsilon}, Y_s^{\varepsilon,u^\varepsilon}) : \nabla_x \nabla_x \chi(X_s^{\varepsilon,u^\varepsilon}, Y_s^{\varepsilon,u^\varepsilon}) \,ds \\
  \notag &+ \frac{\delta}{h(\varepsilon)} \int_0^t \left( \big( \nabla_x \chi(X_s^{\varepsilon,u^\varepsilon}, Y_s^{\varepsilon,u^\varepsilon}) \big) \sigma(X_s^{\varepsilon,u^\varepsilon}, Y_s^{\varepsilon,u^\varepsilon}) + \frac{1}{\delta} \big( \nabla_y \chi(X_s^{\varepsilon,u^\varepsilon}, Y_s^{\varepsilon,u^\varepsilon}) \big) \tau_1(X_s^{\varepsilon,u^\varepsilon}, Y_s^{\varepsilon,u^\varepsilon}) \right) \,dW_s \\
  \notag &+ \frac{1}{h(\varepsilon)} \int_0^t \big( \nabla_y \chi(X_s^{\varepsilon,u^\varepsilon}, Y_s^{\varepsilon,u^\varepsilon}) \big) \tau_2(X_s^{\varepsilon,u^\varepsilon}, Y_s^{\varepsilon,u^\varepsilon}) \,dB_s .
\end{align}

Using Lemmas \ref{L:Ygrowth} and \ref{L:productBound} and Doob's martingale inequality, along with the facts that the integrands that appear in the previous display  grow no more than polynomially in $|y|^{r}$ (Condition \ref{C:Tightness} is being used here), we have
\begin{align*}
  &\E \sup_{t \in [0, 1]} \left\lvert \int_0^t \frac{ \frac{\varepsilon}{\delta} b(X_s^{\varepsilon,u^\varepsilon}, Y_s^{\varepsilon,u^\varepsilon}) + c(X_s^{\varepsilon,u^\varepsilon}, Y_s^{\varepsilon,u^\varepsilon}) - \lambda_1(X_s^{\varepsilon,u^\varepsilon}, Y_s^{\varepsilon,u^\varepsilon}) }{\sqrt{\varepsilon}h(\varepsilon)} \,ds \right. \\
  \notag & \left. - \int_0^t \big( \nabla_y \chi(X_s^{\varepsilon,u^\varepsilon}, Y_s^{\varepsilon,u^\varepsilon}) \big) \left[ \tau_1(X_s^{\varepsilon,u^\varepsilon}, Y_s^{\varepsilon,u^\varepsilon}) u_1^\varepsilon(s) + \tau_2(X_s^{\varepsilon,u^\varepsilon}, Y_s^{\varepsilon,u^\varepsilon}) u_2^\varepsilon(s) \right] \,ds \right\rvert^2 \\
  \notag & < ( \delta C_1)^2 + \left(\frac{1}{h(\varepsilon)} C_2 \right)^2 + o \left( \delta^2 + \frac{1}{h^2(\varepsilon)} \right)
\end{align*}
where  the constants $C_1, C_2 < \infty$ do not depend on $\varepsilon$.

(ii). Separate the integral in \eqref{E:probLim} as in \eqref{E:coefsExpand}, and then most terms go to zero in probability as $\varepsilon$ goes to zero. The only exception is the term
\begin{equation*}
\int_{t_1}^{t_2} \big( \nabla_y \chi(X_s^{\varepsilon,u^\varepsilon}, Y_s^{\varepsilon,u^\varepsilon}) \big) \left[ \tau_1(X_s^{\varepsilon,u^\varepsilon}, Y_s^{\varepsilon,u^\varepsilon}) u_1^\varepsilon(s) + \tau_2(X_s^{\varepsilon,u^\varepsilon}, Y_s^{\varepsilon,u^\varepsilon}) u_2^\varepsilon(s) \right] \,ds .
\end{equation*}

Condition \ref{C:Tightness} and Theorem \ref{T:regularity} imply that Lemma \ref{L:productBound} can be applied, concluding the proof of the statement.

(iii). Rewrite \eqref{E:expBound} as in \eqref{E:coefsExpand} and use the triangle inequality. Take expectations, and for fixed $\varepsilon$, all terms are bounded by Lemmas \ref{L:uBound}, \ref{L:Ygrowth}, and \ref{L:productBound} and Doob's martingale inequality.
\end{proof}

\begin{lemma}\label{L:lambda}
Assume Conditions \ref{C:growth}, \ref{C:ergodic}, \ref{C:Tightness} and \ref{C:center} and define  the processes $X^{\varepsilon, u^\varepsilon}$ and $Y^{\varepsilon, u^\varepsilon}$ by \eqref{E:controlledSDE}. Let $N<\infty$ such that almost surely
\begin{equation*}
  \sup_{\varepsilon > 0}  \int_0^1 \lvert u^\varepsilon(s) \rvert^2 \,ds < N.
\end{equation*}
Also, define the function $\Phi_1(x, y)$ by \eqref{E:Phi}. Then
\begin{enumerate}[(i)]
\item
\begin{align*}
    &\E \sup_{t \in [0,1]} \left\lvert \frac{1}{\sqrt{\varepsilon} h(\varepsilon)} \int_0^t \left( \lambda_1( X_s^{\varepsilon, u^\varepsilon}, Y_s^{\varepsilon, u^\varepsilon} ) - \bar{\lambda}_1(X_s^{\varepsilon, u^\varepsilon}) \right) \,ds  \right. \\
    & \left. {} - \frac{\delta / \varepsilon}{ \sqrt{\varepsilon} h(\varepsilon)} \int_0^t \big( \nabla_y \Phi_1 ( X_s^{\varepsilon, u^\varepsilon}, Y_s^{\varepsilon, u^\varepsilon} ) \big) g( X_s^{\varepsilon, u^\varepsilon}, Y_s^{\varepsilon, u^\varepsilon} ) \,ds \right\rvert^2 \le \left( \frac{\delta}{\sqrt{\varepsilon} h(\varepsilon)} C \right)^2 + o\left(\frac{\delta^2}{\varepsilon h^2(\varepsilon)} \right)
\end{align*}
where the constant $C < \infty$ does not depend on the choice of $\varepsilon$.
\item For every $\zeta > 0$,
\begin{equation*} \notag
  \lim_{\rho \downarrow 0} \limsup_{\varepsilon \downarrow 0} \p \left[ \sup_{\substack{0 \le t_1 < t_2 \le 1 \\ \lvert t_2 - t_1 \rvert < \rho}} \left\lvert \frac{1}{\sqrt{\varepsilon} h(\varepsilon)} \int_{t_1}^{t_2} \left( \lambda_1( X_s^{\varepsilon, u^\varepsilon}, Y_s^{\varepsilon, u^\varepsilon} ) - \bar{\lambda}_1(X_s^{\varepsilon, u^\varepsilon}) \right) ds \right\rvert > \zeta \right] = 0 .
\end{equation*}
\item There exists an $M > 0$ such that for all sufficiently small $\varepsilon$,
\begin{equation*} \notag
  \E\sup_{t\in[0,1]}\left\lvert \frac{1}{\sqrt{\varepsilon} h(\varepsilon)} \int_0^t \left( \lambda_1( X_s^{\varepsilon, u^\varepsilon}, Y_s^{\varepsilon, u^\varepsilon} ) - \bar{\lambda}_1(X_s^{\varepsilon, u^\varepsilon}) \right) ds \right\rvert^2 < M.
\end{equation*}
\end{enumerate}
\end{lemma}

\begin{proof}

(i). Note that
\begin{equation*}
  \int_\mathcal{Y} \left( \lambda_1(x,y) - \bar{\lambda}_1(x) \right) \mu_x (dy) = 0
\end{equation*}
for fixed $x$ by the definition of $\bar{\lambda}_1$ in \eqref{E:Gbar}. Then by Theorem \ref{T:regularity}, \eqref{E:Phi} has a unique smooth solution for every $x$ in the space of functions with at most polynomial growth in $y$.

Apply the It\^{o} formula to $\Phi_1( X_t^{\varepsilon, u^\varepsilon}, Y_t^{\varepsilon, u^\varepsilon} )$
and rearrange to obtain
\begin{align} \label{E:lambdaExpand}
  &\frac{1}{\sqrt{\varepsilon} h(\varepsilon)} \int_0^t \left( \lambda_1( X_s^{\varepsilon, u^\varepsilon}, Y_s^{\varepsilon, u^\varepsilon} ) - \bar{\lambda}_1(X_s^{\varepsilon, u^\varepsilon}) \right) \,ds = - \frac{\delta^2/ \varepsilon}{\sqrt{\varepsilon} h(\varepsilon)} \left( \Phi_1(X_t^{\varepsilon, u^\varepsilon}, Y_t^{\varepsilon, u^\varepsilon}) - \Phi_1 ( x_0, y_0 ) \right) \\
  & + \frac{\delta^2 / \varepsilon}{\sqrt{\varepsilon} h(\varepsilon)}  \int_0^t \big( \nabla_x \Phi_1( X_s^{\varepsilon, u^\varepsilon}, Y_s^{\varepsilon, u^\varepsilon} ) \big) \left( \frac{\varepsilon}{\delta} b( X_s^{\varepsilon, u^\varepsilon}, Y_s^{\varepsilon, u^\varepsilon} ) + c( X_s^{\varepsilon, u^\varepsilon}, Y_s^{\varepsilon, u^\varepsilon} ) \right) \,ds \notag \\
  & + \frac{\delta^2 / \varepsilon}{\sqrt{\varepsilon} h(\varepsilon)}  \int_0^t \frac{\varepsilon}{2} \sigma\sigma^\T( X_s^{\varepsilon, u^\varepsilon}, Y_s^{\varepsilon, u^\varepsilon} ) : \nabla_x\nabla_x \Phi_1( X_s^{\varepsilon, u^\varepsilon}, Y_s^{\varepsilon, u^\varepsilon} ) \,ds \notag \\
   & + \frac{\delta^2}{\varepsilon} \int_0^t \big( \nabla_x \Phi_1( X_s^{\varepsilon, u^\varepsilon}, Y_s^{\varepsilon, u^\varepsilon} ) \big) \sigma( X_s^{\varepsilon, u^\varepsilon}, Y_s^{\varepsilon, u^\varepsilon} ) u_{1}^\varepsilon(s) \,ds \notag \\
    &+ \frac{\delta / \varepsilon}{ \sqrt{\varepsilon} h(\varepsilon)} \int_0^t \big( \nabla_y \Phi_1 ( X_s^{\varepsilon, u^\varepsilon}, Y_s^{\varepsilon, u^\varepsilon} ) \big) g( X_s^{\varepsilon, u^\varepsilon}, Y_s^{\varepsilon, u^\varepsilon} ) \,ds \notag \\
   & + \frac{\delta}{\varepsilon} \int_0^t \big( \nabla_y \Phi_1( X_s^{\varepsilon, u^\varepsilon}, Y_s^{\varepsilon, u^\varepsilon} ) \big) \left[\tau_1( X_s^{\varepsilon, u^\varepsilon}, Y_s^{\varepsilon, u^\varepsilon} ) u_{1}^\varepsilon(s) + \tau_2( X_s^{\varepsilon, u^\varepsilon}, Y_s^{\varepsilon, u^\varepsilon} ) u_{2}^\varepsilon(s) \right] \,ds \notag \\
  & + \frac{\delta^2}{\varepsilon h(\varepsilon)} \int_0^t \big( \nabla_x \Phi_1( X_s^{\varepsilon, u^\varepsilon}, Y_s^{\varepsilon, u^\varepsilon} ) \big) \sigma( X_s^{\varepsilon, u^\varepsilon}, Y_s^{\varepsilon, u^\varepsilon} ) \,dW_s \notag \\
  & + \frac{\delta}{\varepsilon h(\varepsilon)} \int_0^t \big( \nabla_y \Phi_1( X_s^{\varepsilon, u^\varepsilon}, Y_s^{\varepsilon, u^\varepsilon} ) \big) \tau_1( X_s^{\varepsilon, u^\varepsilon}, Y_s^{\varepsilon, u^\varepsilon} ) \,dW_s \notag \\
  & + \frac{\delta}{\varepsilon h(\varepsilon)} \int_0^t \big( \nabla_y \Phi_1( X_s^{\varepsilon, u^\varepsilon}, Y_s^{\varepsilon, u^\varepsilon} ) \big) \tau_2( X_s^{\varepsilon, u^\varepsilon}, Y_s^{\varepsilon, u^\varepsilon} ) \,dB_s . \notag
\end{align}

Due to Conditions \ref{C:growth} and \ref{C:ergodic}, $\lambda_1(x, y) - \bar{\lambda}_1(x)$ satisfies the condition of Theorem \ref{T:regularity}. Notice that a polynomial bound on $\lvert y \rvert$ of $\nabla_x \nabla_x \nabla_y \chi$ is needed. However, due to Condition \ref{C:growth}, this follows by Theorems 1 and 2 in \cite{PV01}.

Using Theorem \ref{T:regularity}, Lemmas \ref{L:Ygrowth} and \ref{L:productBound} and Doob's martingale inequality, along with the facts that the integrands that appear in the previous display  grow no more than polynomially in $|y|^{r}$ (Condition \ref{C:Tightness} is being used here), we have
\begin{align*}
    &\E \sup_{t \in [0,1]} \left\lvert \frac{1}{\sqrt{\varepsilon} h(\varepsilon)} \int_0^t \left( \lambda_1( X_s^{\varepsilon, u^\varepsilon}, Y_s^{\varepsilon, u^\varepsilon} ) - \bar{\lambda}_1(X_s^{\varepsilon, u^\varepsilon}) \right) \,ds  \right. \\
    &\left. {} - \frac{\delta / \varepsilon}{\sqrt{\varepsilon} h(\varepsilon)} \int_0^t \big( \nabla_y \Phi_1 ( X_s^{\varepsilon, u^\varepsilon}, Y_s^{\varepsilon, u^\varepsilon} ) \big) g( X_s^{\varepsilon, u^\varepsilon}, Y_s^{\varepsilon, u^\varepsilon} ) \,ds \right\rvert^2 \le \left( \frac{\delta}{\sqrt{\varepsilon} h(\varepsilon)} C \right)^2 + o\left(\frac{\delta^2}{\varepsilon h^2(\varepsilon)} \right)
\end{align*}
where $C$ does not depend on $\varepsilon$.

(ii). Again using \eqref{E:lambdaExpand}, most terms go to zero in probability as $\varepsilon$ goes to zero. The possible exception is the term
\begin{equation*}
  \frac{\delta / \varepsilon }{\sqrt{\varepsilon} h(\varepsilon)} \int_{t_1}^{t_2} \big( \nabla_y \Phi_1 ( X_s^{\varepsilon, u^\varepsilon}, Y_s^{\varepsilon, u^\varepsilon} ) \big) g( X_s^{\varepsilon, u^\varepsilon}, Y_s^{\varepsilon, u^\varepsilon} ) \,ds .
\end{equation*}

Condition \ref{C:Tightness} and Theorem \ref{T:regularity} imply that Lemma \ref{L:productBound} can be applied, which together with the constraint $\frac{\delta / \varepsilon}{\sqrt{\varepsilon} h(\varepsilon)} \to j_1 < \infty$ conclude the proof of the statement.

(iii). Again, use \eqref{E:lambdaExpand} and use the triangle inequality. For fixed $\varepsilon$, by Lemmas \ref{L:uBound}, \ref{L:Ygrowth}, and \ref{L:productBound}, all terms are bounded in $L^2([0,1] \times \p)$.
\end{proof}

\begin{lemma}\label{L:lambdabar}
Assume Conditions \ref{C:growth}, \ref{C:ergodic}, \ref{C:Tightness} and \ref{C:center}  and define the processes $X^{\varepsilon, u^\varepsilon}$ and $Y^{\varepsilon, u^\varepsilon}$ by \eqref{E:controlledSDE}. Let $N<\infty$ such that almost surely
\begin{equation*}
  \sup_{\varepsilon > 0}  \int_0^1 \lvert u^\varepsilon(s) \rvert^2 \,ds < N.
\end{equation*}
Also, define $\eta^{\varepsilon, u^\varepsilon}$ by \eqref{E:controlledDeviations}. Then
\begin{enumerate}[(i)]
\item
\begin{equation*} \notag
  \E \sup_{t\in[0,1]}\lvert \eta_t^{\varepsilon, u^\varepsilon} \rvert^2 \le K \exp(L_\lambda^2 )
\end{equation*}
where $L_\lambda$ is the Lipschitz constant for $\bar{\lambda}_1$ and the constant $K$ does not depend on $\varepsilon$.
\item  For every $\zeta > 0$,
\begin{equation*} \notag
  \lim_{\rho \downarrow 0} \limsup_{\varepsilon \downarrow 0} \p \left[  \sup_{\substack{0 \le t_1 < t_2 \le 1 \\ \lvert t_2 - t_1 \rvert < \rho}} \left\lvert \frac{1}{\sqrt{\varepsilon} h(\varepsilon)} \int_{t_1}^{t_2}  \left( \bar{\lambda}_1(X_s^{\varepsilon, u^\varepsilon}) - \bar{\lambda}_1(\bar{X}_s) \right) \,ds \right\rvert > \zeta \right] = 0 .
\end{equation*}
\end{enumerate}
\end{lemma}

\begin{proof}

(i). $\bar{\lambda}_1(x)$ is Lipschitz with Lipschitz constant $L_\lambda$ by the fact that its first derivative is bounded. Write \eqref{E:controlledDeviations} as
\begin{equation*}
  X_t^{\varepsilon, u^\varepsilon} = \bar{X}_t + \sqrt{\varepsilon} h(\varepsilon) \eta_t^{\varepsilon, u^\varepsilon}
\end{equation*}
and then
\begin{equation*}
  \lvert \bar{\lambda}_1(X_t^{\varepsilon, u^\varepsilon}) - \bar{\lambda}_1(\bar{X}_t) \rvert \le L_\lambda \lvert \sqrt{\varepsilon} h(\varepsilon) \eta_t^{\varepsilon, u^\varepsilon} \rvert .
\end{equation*}

Therefore we have
\begin{equation} \label{E:lambdabarBound}
  \sup_{0\le t\le 1} \left\lvert \frac{1}{\sqrt{\varepsilon} h(\varepsilon)} \int_0^t \left( \bar{\lambda}_1(X_s^{\varepsilon, u^\varepsilon}) - \bar{\lambda}_1(\bar{X}_s) \right) \,ds \right\rvert^2 \le L_\lambda^2 \int_0^1 \left\lvert \eta_{s}^{\varepsilon, u^\varepsilon} \right\rvert^2 \,ds.
\end{equation}

Using the decomposition in \eqref{E:deviationsIntegral} and \eqref{E:lambdaSplit}, we have up to some multiplicative constant $C<\infty$
\begin{align*}
  \lvert \eta_t^{\varepsilon, u^\varepsilon} \rvert^2 &\le \left\lvert \int_0^t \frac{ \frac{\varepsilon}{\delta} b(X_s^{\varepsilon, u^\varepsilon}, Y_s^{\varepsilon, u^\varepsilon}) + c(X_s^{\varepsilon, u^\varepsilon}, Y_s^{\varepsilon, u^\varepsilon}) - \lambda_1(X_s^{\varepsilon, u^\varepsilon}, Y_s^{\varepsilon, u^\varepsilon})}{\sqrt{\varepsilon} h(\varepsilon)} \,ds \right\rvert^2 \\
  &+ \left\lvert \int_0^t \frac{ \lambda_1(X_s^{\varepsilon, u^\varepsilon}, Y_s^{\varepsilon, u^\varepsilon}) - \bar{\lambda}_1(X_s^{\varepsilon, u^\varepsilon})}{\sqrt{\varepsilon} h(\varepsilon)} \,ds \right\rvert^2
  + \left\lvert \int_0^t \frac{\bar{\lambda}_1(X_s^{\varepsilon, u^\varepsilon}) - \bar{\lambda}_1(\bar{X}_s)}{\sqrt{\varepsilon} h(\varepsilon)} \,ds \right\rvert^2 \notag \\
  &+ \left\lvert \int_0^t \sigma(X_s^{\varepsilon, u^\varepsilon}, Y_s^{\varepsilon, u^\varepsilon}) u_1^\varepsilon(s)\, ds \right\rvert^2
   + \left\lvert \int_0^t \frac{1}{h(\varepsilon)} \sigma(X_s^{\varepsilon, u^\varepsilon}, Y_s^{\varepsilon, u^\varepsilon}) \,dW_s \right\rvert^2 . \notag
\end{align*}

Take supremum in $t\in[0,1]$ and then  expectations of both sides. For sufficiently small $\varepsilon$, the first term is bounded by Lemma \ref{L:coefs}. The second term is bounded by Lemma \ref{L:lambda}. The third term is bounded by \eqref{E:lambdabarBound}. The fourth term is bounded by Lemma \ref{L:productBound}. Finally, the expectation of the fifth term is bounded due to Doob's martingale inequality and the bound on $\sigma$ together with Lemma \ref{L:Ygrowth}. Combining these, we get
\begin{equation*}
  \E \sup_{t\in[0,1]}\lvert \eta_t^{\varepsilon, u^\varepsilon} \rvert^2 \le K + L_\lambda^2 \int_0^1 \E \sup_{s\in[0,t]} \left\lvert \eta_{s}^{\varepsilon, u^\varepsilon} \right\rvert^2 \,dt
\end{equation*}
where $K$ is the sum of the bounds on the expectations of terms one, two, four, and five.  Then by Gronwall's lemma, we have the required statement,
\begin{equation} \label{E:etaBound}
  \E \sup_{t\in[0,1]}\lvert \eta_t^{\varepsilon, u^\varepsilon} \rvert^2 \le K \exp(L_\lambda^2 )
\end{equation}
for all sufficiently small $\varepsilon>0$.

(ii) From \eqref{E:lambdabarBound} and the Markov inequality,
\begin{multline*}
  \p \left[  \sup_{\substack{0 \le t_1 < t_2 \le 1 \\ \lvert t_2 - t_1 \rvert < \rho}} \left\lvert \frac{1}{\sqrt{\varepsilon} h(\varepsilon)} \int_{t_1}^{t_2}  \left( \bar{\lambda}_1(X_s^{\varepsilon, u^\varepsilon}) - \bar{\lambda}_1(\bar{X}_s) \right) \,ds \right\rvert^2 > \zeta^2 \right] \\
  \le \frac{L_\lambda^2}{\zeta^2} \E \sup_{\substack{0 \le t_1 < t_2 \le 1 \\ \lvert t_2 - t_1 \rvert < \rho}} \int_{t_1}^{t_2}  \left\lvert \eta_{s}^{\varepsilon, u^\varepsilon} \right\rvert^2 \,ds .
\end{multline*}
Since $\E \sup_{t\in[0,1]}\lvert \eta_t^{\varepsilon, u^\varepsilon} \rvert^2$ is uniformly bounded by \eqref{E:etaBound} for $\varepsilon$ small enough, this probability goes to zero as $\lvert t_2 - t_1 \rvert$ goes to zero, completing the proof.

\end{proof}

\section{Lemmas for Regime 2}

Notice that Lemmas \ref{L:uBound}, \ref{L:Ygrowth} and \ref{L:productBound} are also valid for Regime 2. Statements and proofs for the lemmas corresponding to Lemmas \ref{L:coefs}, \ref{L:lambda} and \ref{L:lambdabar} are similar to those in Regime 1, by considering $\lambda_2$ in place of $\lambda_1$. The only difference is in the proof of the statement that corresponds to Lemma \ref{L:lambda}(i), which we now state and prove.

\begin{lemma}
Assume Conditions \ref{C:growth}, \ref{C:ergodic}, and \ref{C:Tightness} and define the processes $X^{\varepsilon, u^\varepsilon}$ and $Y^{\varepsilon, u^\varepsilon}$ by \eqref{E:controlledSDE}. Let $N<\infty$ such that almost surely
\begin{equation*}
  \sup_{\varepsilon > 0}  \int_0^1 \lvert u^\varepsilon(s) \rvert^2 \,ds < N.
\end{equation*}
Also, define the function $\Phi_2(x, y)$ by \eqref{E:Phi} with $j_2 < \infty$. Then
\begin{align*}
    &\E \sup_{t \in [0,1]} \left\lvert \frac{1}{\sqrt{\varepsilon} h(\varepsilon)} \int_0^t \left( \lambda_2( X_s^{\varepsilon, u^\varepsilon}, Y_s^{\varepsilon, u^\varepsilon} ) - \bar{\lambda}_2(X_s^{\varepsilon, u^\varepsilon}) \right) \,ds  \right. \\
    &\left. - \int_0^t \big( \nabla_y \Phi_2( X_s^{\varepsilon, u^\varepsilon}, Y_s^{\varepsilon, u^\varepsilon} ) \big) \left[ \tau_1( X_s^{\varepsilon, u^\varepsilon}, Y_s^{\varepsilon, u^\varepsilon} ) u_{1}^\varepsilon(s) + \tau_2( X_s^{\varepsilon, u^\varepsilon}, Y_s^{\varepsilon, u^\varepsilon} ) u_{2}^\varepsilon(s) \right] \,ds \right. \notag \\
    &\left. - \frac{\varepsilon / \delta - \gamma}{\sqrt{\varepsilon} h(\varepsilon) } \int_0^t \big( \nabla_y \Phi_2( X_s^{\varepsilon, u^\varepsilon}, Y_s^{\varepsilon, u^\varepsilon} ) \big) f( X_s^{\varepsilon, u^\varepsilon}, Y_s^{\varepsilon, u^\varepsilon} ) \,ds \right. \notag \\
    &\left. - \frac{1}{2} \frac{\varepsilon / \delta - \gamma}{\sqrt{\varepsilon} h(\varepsilon) } \int_0^t \left(\tau_1\tau_1^\T + \tau_2\tau_2^\T \right)( X_s^{\varepsilon, u^\varepsilon}, Y_s^{\varepsilon, u^\varepsilon} ) : \nabla_y\nabla_y \Phi_2( X_s^{\varepsilon, u^\varepsilon}, Y_s^{\varepsilon, u^\varepsilon} ) \,ds \right\rvert^2 \notag \\
    &\le \left( \delta C_1 \right)^2 + \left( \frac{1}{h(\varepsilon)} C_2 \right)^2 + o\left( \delta^2 + \frac{1}{h^2(\varepsilon)} \right) \notag
\end{align*}
where the constants $C_1$ and $C_2$ do not depend on the choice of $\varepsilon$.
\end{lemma}

\begin{proof}
Note that in Regime $2$
\begin{equation*}
  \int_\mathcal{Y} \left( \lambda_2(x,y) - \bar{\lambda}_2(x) \right) \mu_{2, x} (dy) = 0
\end{equation*}
for fixed $x$ by the definition of $\bar{\lambda}_2$ in \eqref{E:Gbar}. Then \eqref{E:Phi} has a unique, smooth solution for every $x$ that is bounded in $x$ and grows at most polynomially in $\lvert y \rvert$ as in Theorem \ref{T:regularity}.

Apply the It\^{o} formula to $\Phi_2( X_t^{\varepsilon, u^\varepsilon}, Y_t^{\varepsilon, u^\varepsilon} )$
and rearrange to show
\begin{align*} \label{E:lambda2Expand}
  &\frac{1}{\sqrt{\varepsilon} h(\varepsilon)} \int_0^t \left( \lambda_2( X_s^{\varepsilon, u^\varepsilon}, Y_s^{\varepsilon, u^\varepsilon} ) - \bar{\lambda}_2(X_s^{\varepsilon, u^\varepsilon}) \right) \,ds = - \frac{\delta}{\sqrt{\varepsilon} h(\varepsilon)} \left( \Phi_2(X_t^{\varepsilon, u^\varepsilon}, Y_t^{\varepsilon, u^\varepsilon}) - \Phi_2 ( x_0, y_0 ) \right) \\
  &+ \frac{\delta}{\sqrt{\varepsilon} h(\varepsilon)} \int_0^t \big( \nabla_x \Phi_2( X_s^{\varepsilon, u^\varepsilon}, Y_s^{\varepsilon, u^\varepsilon} ) \big) \left( \frac{\varepsilon}{\delta} b( X_s^{\varepsilon, u^\varepsilon}, Y_s^{\varepsilon, u^\varepsilon} ) + c( X_s^{\varepsilon, u^\varepsilon}, Y_s^{\varepsilon, u^\varepsilon} ) \right) \,ds \notag \\
  &+  \frac{\delta \sqrt{\varepsilon} }{h(\varepsilon)} \frac{1}{2} \int_0^t \sigma\sigma^\T( X_s^{\varepsilon, u^\varepsilon}, Y_s^{\varepsilon, u^\varepsilon} ) : \nabla_x\nabla_x \Phi_2( X_s^{\varepsilon, u^\varepsilon}, Y_s^{\varepsilon, u^\varepsilon} ) \,ds \notag\\
  &+ \delta \int_0^t \big( \nabla_x \Phi_2( X_s^{\varepsilon, u^\varepsilon}, Y_s^{\varepsilon, u^\varepsilon} ) \big) \sigma( X_s^{\varepsilon, u^\varepsilon}, Y_s^{\varepsilon, u^\varepsilon} ) u_{1}^\varepsilon(s) \,ds \notag\\
    &+  \int_0^t \big( \nabla_y \Phi_2( X_s^{\varepsilon, u^\varepsilon}, Y_s^{\varepsilon, u^\varepsilon} ) \big) \left[ \tau_1( X_s^{\varepsilon, u^\varepsilon}, Y_s^{\varepsilon, u^\varepsilon} ) u_{1}^\varepsilon(s) + \tau_2( X_s^{\varepsilon, u^\varepsilon}, Y_s^{\varepsilon, u^\varepsilon} ) u_{2}^\varepsilon(s) \right] \,ds \notag \\
    &+ \frac{\delta}{h(\varepsilon)} \int_0^t \big( \nabla_x \Phi_2( X_s^{\varepsilon, u^\varepsilon}, Y_s^{\varepsilon, u^\varepsilon} ) \big) \sigma( X_s^{\varepsilon, u^\varepsilon}, Y_s^{\varepsilon, u^\varepsilon} ) \,dW_s \notag \\
    &+ \frac{1}{h(\varepsilon)}  \int_0^t \big( \nabla_y \Phi_2( X_s^{\varepsilon, u^\varepsilon}, Y_s^{\varepsilon, u^\varepsilon} ) \big) \tau_1( X_s^{\varepsilon, u^\varepsilon}, Y_s^{\varepsilon, u^\varepsilon} ) \,dW_s \notag \\
    &+  \frac{1}{h(\varepsilon)} \int_0^t \big( \nabla_y \Phi_2( X_s^{\varepsilon, u^\varepsilon}, Y_s^{\varepsilon, u^\varepsilon} ) \big) \tau_2( X_s^{\varepsilon, u^\varepsilon}, Y_s^{\varepsilon, u^\varepsilon} ) \,dB_s \notag \\
    &+ \frac{\varepsilon / \delta - \gamma}{\sqrt{\varepsilon} h(\varepsilon) } \int_0^t \big( \nabla_y \Phi_2( X_s^{\varepsilon, u^\varepsilon}, Y_s^{\varepsilon, u^\varepsilon} ) \big) f( X_s^{\varepsilon, u^\varepsilon}, Y_s^{\varepsilon, u^\varepsilon} ) \,ds \notag \\
    &+ \frac{1}{2} \frac{\varepsilon / \delta - \gamma}{\sqrt{\varepsilon} h(\varepsilon) } \int_0^t \left(\tau_1\tau_1^\T + \tau_2\tau_2^\T \right)( X_s^{\varepsilon, u^\varepsilon}, Y_s^{\varepsilon, u^\varepsilon} ) : \nabla_y\nabla_y \Phi_2( X_s^{\varepsilon, u^\varepsilon}, Y_s^{\varepsilon, u^\varepsilon} ) \,ds . \notag
\end{align*}

Using Theorem \ref{T:regularity}, Lemmas \ref{L:Ygrowth} and \ref{L:productBound} and Doob's martingale inequality, along with the facts that the integrands that appear in the previous display  grow no more than polynomially in $|y|^{r}$ (Condition \ref{C:Tightness} is being used here), we have
\begin{align*}
    &\E \sup_{t \in [0,1]} \left\lvert \frac{1}{\sqrt{\varepsilon} h(\varepsilon)} \int_0^t \left( \lambda_2( X_s^{\varepsilon, u^\varepsilon}, Y_s^{\varepsilon, u^\varepsilon} ) - \bar{\lambda}_2(X_s^{\varepsilon, u^\varepsilon}) \right) \,ds  \right. \\
    &\left. - \int_0^t  \big( \nabla_y \Phi_2( X_s^{\varepsilon, u^\varepsilon}, Y_s^{\varepsilon, u^\varepsilon} ) \big) \left[ \tau_1( X_s^{\varepsilon, u^\varepsilon}, Y_s^{\varepsilon, u^\varepsilon} ) u_{1}^\varepsilon(s) + \tau_2( X_s^{\varepsilon, u^\varepsilon}, Y_s^{\varepsilon, u^\varepsilon} ) u_{2}^\varepsilon(s) \right] \,ds \right. \notag \\
    &\left. - \frac{\varepsilon / \delta - \gamma}{\sqrt{\varepsilon} h(\varepsilon) } \int_0^t \big( \nabla_y \Phi_2( X_s^{\varepsilon, u^\varepsilon}, Y_s^{\varepsilon, u^\varepsilon} ) \big) f( X_s^{\varepsilon, u^\varepsilon}, Y_s^{\varepsilon, u^\varepsilon} ) \,ds \right. \notag \\
    &\left. - \frac{1}{2} \frac{\varepsilon / \delta - \gamma}{\sqrt{\varepsilon} h(\varepsilon) } \int_0^t \left(\tau_1\tau_1^\T + \tau_2\tau_2^\T \right)( X_s^{\varepsilon, u^\varepsilon}, Y_s^{\varepsilon, u^\varepsilon} ) : \nabla_y\nabla_y \Phi_2( X_s^{\varepsilon, u^\varepsilon}, Y_s^{\varepsilon, u^\varepsilon} ) \,ds \right\rvert^2 \notag \\
    &\le \left( \delta C_1 \right)^2 + \left( \frac{1}{h(\varepsilon)} C_2 \right)^2 + o\left( \delta^2 + \frac{1}{h^2(\varepsilon)} \right) \notag
\end{align*}
where $C_1$ and $C_2$ do not depend on $\varepsilon$.
\end{proof}

\end{document}